\documentclass[10pt]{amsart}
\usepackage[T1]{fontenc}            
\usepackage[utf8]{inputenc}         
\usepackage[english]{babel} 
\usepackage{physics}
\usepackage{latexsym}
\usepackage{amsmath}
\usepackage{amssymb}
\usepackage{mathrsfs}
\usepackage{graphicx}
\usepackage{color}
\usepackage{pgfpages}
\usepackage{ifthen}
\usepackage{leftidx,tensor}
\usepackage{mathtools}
\usepackage{comment}
\usepackage{dsfont}
\usepackage[nocompress]{cite}

\usepackage[shortlabels]{enumitem}
\usepackage{aliascnt}
\usepackage[bookmarks=true,pdfstartview=FitH, pdfborder={0 0 0}, colorlinks=true,citecolor=red, linkcolor=blue]{hyperref}
\usepackage{bbm}
\usepackage{nicefrac}

\newtheorem{theorem}{Theorem}[section]

\newtheorem{corollary}[theorem]{Corollary}

\newtheorem{definition}[theorem]{Definition}
\newtheorem{example}[theorem]{Example}

\newtheorem{lemma}[theorem]{Lemma}

\newtheorem{proposition}[theorem]{Proposition}
\newtheorem{remark}[theorem]{Remark}

\numberwithin{equation}{section}

\setlist[enumerate]{font = \normalfont}

\newcommand {\R}	{\mathbb{R}}

\renewcommand{\Re}{\mathrm{Re}}

	\renewcommand{\phi}{\varphi}
	
\title[A multiplicative-noise mechanism for variability amplification in an Arctic EBM]{A Multiplicative-Noise Mechanism for Variability Amplification under Radiative Forcing in an Arctic Energy-Balance Model}

  \author{Gianmarco Del Sarto}
\address{Technische Universit\"{a}t Darmstadt\\
Fachbereich Mathematik\\
	Schlossgartenstr.\ 7\\
	64289 Darmstadt\\
	Germany\\
    }
\email{delsarto@mathematik.tu-darmstadt.de}

\author{Franco Flandoli}
\address{Scuola Normale Superiore\\
	Classe di Scienze\\
	P.za dei Cavalieri 7\\
	56126 Pisa\\
	Italy}
\email{franco.flandoli@sns.it}
\author{Marta Lenzi}
\address{Scuola Normale Superiore\\
	Classe di Scienze\\
	P.za dei Cavalieri 7\\
	56126 Pisa\\
	Italy}
\email{marta.lenzi@sns.it}
   \address{University School for Advanced Studies IUSS Pavia \\
Department of Science, Technology and Society\\
P.za della Vittoria 15\\
27100 Pavia\\
Italy
}
\email{marta.lenzi@iusspavia.it}

\begin{document}

\maketitle

\begin{abstract}
    We propose and analyse a mechanism by which $\mathrm{CO}_2$-driven radiative forcing can increase Arctic temperature variability in a stochastic Sellers-type energy-balance model. Starting from a fast-slow formulation in which insolation is modelled by a rapidly mean-reverting Ornstein-Uhlenbeck process while temperature evolves on a slow macroweather timescale, a Wong-Zakai reduction leads to a stochastic energy-balance equation with \emph{multiplicative} noise. After linearising around the stable equilibrium $T^{*,\lambda}$, we derive an explicit expression for the stationary variance of the temperature anomaly and prove that it increases monotonically with the forcing parameter $\lambda$ whenever $T^{*,\lambda}$ lies in the ice-sensitive regime of the co-albedo. We then consider a spatial anomaly model and its finite-difference semi-discretisation, obtaining a finite-dimensional SDE. Under natural stability conditions and nonnegative noise correlations, we establish a component-wise monotone increase of the stationary covariance matrix with respect to $\lambda$, including its off-diagonal entries. In particular, radiative forcing amplifies not only local variances but also the covariance between temperature anomalies at distinct spatial locations, indicating increased similarity in the variability of the anomaly field across space.
\end{abstract}
\tableofcontents

\section*{Introduction}

\subsubsection*{Background and motivation.} The Arctic is warming significantly faster than the global average, a phenomenon known as Arctic Amplification \cite{Serreze}. A primary driver of this process is the ice–albedo feedback: as snow and sea ice melt, the surface darkens and absorbs more shortwave radiation \cite{Hall,Winton}. Extreme events are controlled not only by shifts in the mean, but also by changes in variability. In particular, an increase in variance widens the distribution and makes persistent or extreme warm events more likely \cite{IPCC2001,IPCC2021}. This paper proposes and analyses, within a mathematically tractable framework, a mechanism by which
$\mathrm{CO_2}$-driven radiative forcing can \emph{increase Arctic temperature variability} through the
temperature-dependent ice-albedo feedback. 

\subsubsection*{Modelling framework and main results}
We work within the class of energy balance models (EBMs), an elementary class of climate models \cite{North17}. In their simplest form, EBMs describe
the evolution of a representative near-surface temperature through the balance between absorbed shortwave
radiation and emitted longwave radiation \cite{Budyko1969,Sellers1969,Ghil1976}. EBMs are appreciated for their mathematical tractability \cite{Diaz97,Floridia2020,DelSarto26b, Fornasaro}, and for capturing key mechanisms driving Earth's climate \cite{Cannarsa2023,Bastiaansen2022, Goodwin, DelSarto26a}. A central ingredient in Sellers-type
EBMs is a \emph{Lipschitz} co-albedo function $\beta(T)$ encoding the reduction of albedo with warming as sea
ice and snow retreat \cite{North1981,RomboutsGhil2015, Lohmann}. This co-albedo nonlinearity is the simplest
representation of the ice-albedo feedback and is known to generate multiple regimes and tipping-like
behaviour in low-order climate models \cite{Ashwin2020,Ghil2020,lucarini2022}. In the present paper, rather than focusing on regime shifts, we
address how radiative forcing modifies the \emph{stationary variance and covariance} of temperature anomalies around a
stable equilibrium.

A standard approach to modelling unresolved ``weather'' variability in climate dynamics is to introduce
stochastic forcing, in the spirit of Hasselmann-type stochastic climate models \cite{Hasselmann1976,Arnold2001,Watkins}. In EBMs, noise is often
assumed to be additive. However, this is a simplifying assumption, as already pointed out in the Hasselmann programme, where the correct form of the noise is the \emph{multiplicative} one \cite{Hasselmann1976}. This motivates our starting point: a fast-slow formulation in which insolation (or a proxy) evolves
on a fast ``weather'' timescale (e.g. one day) while temperature evolves on a slow macroweather timescale (e.g. one year). A related fast-slow insolation–temperature setup was recently considered in \cite{Longo26}; here we focus on the reduction to a closed equation for the slow variable and the resulting monotone variance/covariance amplification.

More precisely, we consider a coupled fast-slow system in which insolation is modelled by a rapidly
mean-reverting Ornstein-Uhlenbeck process with correlation time $\tau\ll 1$, while temperature obeys a
radiative balance law with co-albedo feedback. We prove that the integrated
fast fluctuations converge to Brownian motion as $\tau\to 0$. Since the integrated forcing is smooth for
fixed $\tau>0$, this naturally leads to a Wong-Zakai reduction: replacing the smooth driver by its Brownian
limit yields a stochastic EBM with \emph{multiplicative noise}. This provides a derivation which links fast weather variability to a reduced stochastic climate model with state-dependent noise.

Our main results quantify a forcing-induced \emph{variance amplification mechanism} in this setting. After
linearising the reduced multiplicative-noise EBM around a stable equilibrium $T^{*,\lambda}$ (depending on
the radiative forcing parameter $\lambda$), we obtain an explicit formula for the stationary variance of the
temperature anomaly. We then prove that this stationary variance increases monotonically with $\lambda$
whenever $T^{*,\lambda}$ lies in the \emph{ice-sensitive} regime of the co-albedo (i.e. where $\beta$ has
positive slope). The mechanism is clear. Increasing radiative forcing shifts the equilibrium
temperature upward; in the ice-sensitive regime the co-albedo is increasing, so the effective amplitude of
multiplicative noise and its impact on the anomaly dynamics increase accordingly, leading to higher
stationary variance.

We then move from a single temperature variable to a spatial model that accounts for horizontal heat transport. We study the temperature anomaly around a reference equilibrium profile and approximate the resulting stochastic dynamics on a grid. This leads to a finite-dimensional stochastic model for the anomaly field, for which we can analyse the long-time covariance explicitly. Under standard stability assumptions and non-negative noise correlations, we show that increasing radiative forcing increases the stationary covariance matrix component-wise, and hence increases the stationary spatial variance (its trace), which we use as a second-moment proxy for fluctuation intensity and for the likelihood of large anomalies. This covariance-amplification prediction is complementary to recent harmonic-domain statistical approaches for time-dependent temperature fields on the sphere, where changes (or nonstationarity) in the dependence structure are investigated via the time evolution of spherical-harmonic coefficients and related change-point/structural-break methodology; see \cite{Caponera2021,Caponera2023,Marinucci2021,Marinucci2024}.
\subsubsection*{Applicability and limitations} Our results are conditional on the existence of a stable deterministic temperature equilibrium
$T^{*,\lambda}$ around which we linearise and compute stationary second moments. For Arctic applications this
assumption is not automatic: it is essentially guaranteed in polar winter (when mean insolation is close to
zero), but it can be questionable if the forcing is interpreted as summer or annual-mean, when the
ice-albedo feedback is strongest. We nevertheless adopt stability as a first-order working hypothesis,
since key EBM parameters are uncertain and hard to calibrate regionally (i.e. in our Arctic setting), and because a more complete treatment would incorporate seasonality by prescribing a periodically time-dependent insolation. This would lead to a non-autonomous problem where the reference equilibrium $T^{*,\lambda}(t)$(time dependent) describes a seasonal cycle rather than a fixed equilibrium. We leave this extension to
future work.

\subsubsection*{Observational comparison} In the last part of the paper we compare our theoretical results with observations, asking whether the model-predicted variability changes can be seen in real data. Since the ice-albedo feedback is expected to be strongest when the mean state lies near the melt-sensitive temperature range, we focus on August \emph{sea surface temperature} in the Arctic and adjacent northern high latitudes and track changes in both the mean and variability over the satellite era. More generally, variance changes under warming can depend strongly on the region \cite{Lewis2017}; here this suggests that variability changes should be spatially non-uniform, with Arctic seas and transition zones behaving differently from persistently ice-covered or warm open-ocean areas.

\subsubsection*{Organisation of the paper} 
This paper is organised as follows. Section \ref{sec: one} introduces the stochastic Sellers-type EBM, derives the multiplicative-noise reduced
model via a fast-slow formulation and Wong-Zakai, and proves monotone variance amplification with respect
to radiative forcing. Section \ref{sec: two} formulates the spatial anomaly model, constructs a finite
difference semi-discretisation, derives covariance dynamics and its stationary limit, and establishes
monotonicity of the stationary covariance and spatial variance under suitable assumptions. Lastly, Section \ref{sec: three} compares the model predictions with observational August \emph{sea surface temperature} from the National Oceanic and Atmospheric Administration (NOAA) north of $60^\circ$N. We track temporal variability with moving decadal statistics and spatial heterogeneity with the within-region standard deviation, contrasting Arctic seas with regions showing little change.

\section{Radiative forcing amplifies variance in a multiplicative-noise 0D-EBM}
\label{sec: one}

\subsection{Preliminaries: energy balance models and Wong-Zakai principle}
\label{subsec: preliminaries}
\subsubsection{Energy balance climate models}
Among the lowest-complexity levels in the climate-model hierarchy lie the class of \emph{energy balance models}
\cite{North17}. Introduced independently by Budyko and Sellers in $1969$ \cite{Budyko1969,Sellers1969}, EBMs describe the evolution of a representative (near-surface) temperature
$T$ (in Kelvin) through the balance between absorbed shortwave radiation $\mathcal{R}_a$ and emitted
longwave radiation $\mathcal{R}_e$. In its zero-dimensional form, the EBM is given by
\begin{equation}
C_{\mathrm{eff}}\,\frac{dT}{dt}=\mathcal{R}_a-\mathcal{R}_e,
    \label{eq: EBM preliminaries}
\end{equation}
where $C_{\mathrm{eff}}>0$ is an effective heat capacity per unit area; equivalently, upon rescaling time one may set
$C_{\mathrm{eff}}=1$ and recover $\frac{dT}{dt}=\mathcal{R}_a-\mathcal{R}_e$. The model can be enriched by incorporating
a spatial variable (typically latitude, after zonal averaging) and representing meridional heat transport
by diffusion, which yields a nonlinear parabolic PDE, see \cite{Ghil1976,North1981}.

In this work we adopt a reduced, regionally EBM for the Arctic. This choice is motivated by the
fact that Arctic temperature variability and trends are strongly constrained by the local radiative energy
budget and amplified by cryospheric feedbacks, most notably the sea-ice/snow albedo feedback, which makes
reduced energy-balance descriptions informative at regional scales \cite{vandeWal,Dortmans}. Consistently, low-order energy-balance and sea-ice models have been adapted for recent studies for sea-ice stability, Arctic warming and ice-loss transitions \cite{Goodwin,Rose,Eisenman}.

For the emitted radiation $\mathcal{R}_e$ we use the classical Budyko empirical formula \cite{Budyko1969,North17} and assume
$$
\mathcal{R}_e(T)=r_0+r_1T,
$$ 
with $r_0 \in \mathbb{R}$ and $r_1 >0$. For the absorbed shortwave radiation we set
$$
\mathcal{R}_a(T,Q,\lambda )=Q \beta(T)+ \lambda ,
$$
where $Q$ denotes the (possibly regionally averaged) incoming solar flux and $\beta\in(0,1)$ is the
\emph{co-albedo}, i.e. the fraction of incident shortwave radiation absorbed ($\beta=1-\alpha$, with
$\alpha$ the albedo). Here, as in \cite{DelSartoNPG,Bastiaansen2022}, $\lambda$ models the effect of greenhouse gases on the radiation balance; the larger is $\lambda$, the larger is the radiative forcing, and thus the greenhouse effect which leads to a higher absorbed radiation. Following continuous piece-wise linear parametrisation in the EBM literature as in \cite{RomboutsGhil2015}, we consider
\begin{equation}
\beta(T) =\left \lbrace
    \begin{aligned}
    &\beta_{\mathrm{min}}, &\qquad & \text{ if } T \leq T_{l},\\
    &\beta_{\mathrm{min}} + \frac{\beta_{ \mathrm{max}}- \beta_{ \mathrm{min}}}{T_u - T_l} (T- T_l), &\qquad &\text{ if } T_l < T< T_{u}, \\
    &\beta_{ \mathrm{max}}, &\qquad & \text{ if } T \geq T_{u}.
    \end{aligned}
    \right.
    \label{eq: piecewise co-albedo}
\end{equation}
where $\beta_{\mathrm{min}}$ and $\beta_{\mathrm{max}}$ represent the limiting co-albedo values in, respectively, the ice-covered (cold) and ice-free
(warm) regimes; typical values in Sellers-type models are $\beta_{\min}\approx 0.38$ and $\beta_{\max}\approx
0.70$ \cite{North1981}. Here, $T_l = 263\,\mathrm{K}$ and $T_u = 300\,\mathrm{K}$ as in \cite{RomboutsGhil2015}. 
We emphasise that these values are not meant to represent the melting temperature of ice, but rather to parametrise the temperature range over which the \emph{effective} (regionally averaged) ice cover, and hence the co-albedo, transitions from its ice-covered to its ice-free limit \cite{North1981,RomboutsGhil2015}. 
This interpretation is consistent with the thermodynamics of melting. Since the phase change involves the latent heat of fusion, a large amount of energy is required to melt the ice, and the temperature can remain close to the freezing point while melting proceeds. In simple sea-ice/EBM formulations this “latent-heat buffering” is captured by treating melt in terms of an enthalpy (energy) budget when ice is present, see \cite{Eisenman,Eisenman2012}. Lastly, note that our EBM is of Sellers-type, meaning that the co-albedo is Lipschitz continuous. The other main class of EBMs is the Budyko type, which is characterised by a discontinuous co-albedo. In the latter case, it is known that uniqueness may fail; see, for example, \cite{Diaz97} for a discussion of non-uniqueness of weak solutions.

\subsubsection{Wong-Zakai principle.} It is worth recalling the Wong-Zakai principle, a tool that we will use in Section \ref{subsec: fast slow model}. Roughly speaking, it asserts that if we consider smooth approximations $W^\varepsilon_t$ of a Brownian motion $W_t$ such that
$$
W^\varepsilon_t \to W_t
$$
in a sufficiently strong sense, then the solution $Z^\varepsilon_t$ of the corresponding (pathwise deterministic, for fixed $\omega$) differential equation
$$
\frac{dZ^\varepsilon_t}{dt}= b(t,Z^\varepsilon_t)+\sigma(t,Z^\varepsilon_t) \frac{dW^\varepsilon_t}{dt},
$$
converges (in the same mode of convergence as $W^\varepsilon_t\to W_t$), as $\varepsilon\to 0$, to the solution of the SDE driven by \emph{Stratonovich noise}
$$
dZ_t=b(t,Z_t)\,dt+\sigma(t,Z_t)\circ dW_t.
$$
See, for instance, \cite[Chapter 5.2]{Karatza1991} or \cite{IkedaWatanabe1981,Twardowska1996} for more details on the Wong-Zakai principle.

For the sake of clarity, we recall that the Stratonovich stochastic integral $\int_0^t \sigma(s,Z_s)\circ dW_s$ is defined as the limit (in probability) of Riemann sums using midpoint evaluation,
$$
\sum_{k}\sigma \left(t_k,\frac{Z_{t_k}+Z_{t_{k+1}}}{2}\right)\left(W_{t_{k+1}}-W_{t_k}\right),
$$
whenever the limit exists. This contrasts with the It\^{o} integral $\int_0^t \sigma(s,Z_s)\,dW_s$, which is defined via left-point evaluation $\sigma(t_k,Z_{t_k})$. A key consequence is that Stratonovich calculus obeys the classical chain rule (as in ordinary calculus), while It\^{o} calculus involves an additional quadratic-variation correction term. Concretely, for sufficiently regular coefficients one has the conversion formula
$$
\int_0^t \sigma(s,Z_s)\circ dW_s = \int_0^t \sigma(s,Z_s)\, dW_s
+\frac{1}{2}\int_0^t \partial_z\sigma(s,Z_s) \sigma(s,Z_s)\,ds,
$$
so that an SDE written in Stratonovich form corresponds to an It\^{o} SDE with an extra correction term in the drift term, see \cite[Section 4.6]{Solin2019}.

\subsection{Fast–slow insolation–temperature model and reduction to a multiplicative-noise EBM}
\label{subsec: fast slow model}
We consider a fast-slow system describing the interaction between a \emph{slow} temperature process $T$ evolving on a macroweather timescale (e.g. one year) and a \emph{fast} component $X$ evolving on a weather timescale (e.g. one day). In our setting, $X$ represents (a proxy for) the insolation. We model $X$ as a mean-reverting Ornstein-Uhlenbeck process fluctuating around a climatological mean level $Q$, while the temperature follows a radiative balance law. More precisely, on the slow time variable $t$ we consider
\begin{equation}
\left \lbrace 
\begin{aligned}
        dX_t^\tau  & = - \frac{1}{\tau } (X_t^\tau- Q) dt + \frac{1}{\sqrt{\tau }}dW_t, \qquad &X_0^\tau &= x_0,\\
    \frac{dT^{\tau,\lambda}_t }{dt} &=  X_t^\tau \beta(T^{\tau,\lambda}_t) + \lambda - (r_0 + r_1 T^{\tau,\lambda}_t) ,  \qquad &T_0^{\tau,\lambda} & = \theta_0,
\end{aligned}
\right.
    \label{eq: fast-slow system}
\end{equation}
where $(W_t)_t$ denotes a Brownian motion, and $x_0, \theta_0$ are the deterministic initial conditions. The parameter $\tau >0$ represents the ratio of the fast and slow timescales (e.g. $\tau \approx 1 /365$ if the slow scale is one year and the fast one is one day). Even if $\tau$ is small, it is not infinitesimal in applications; this will be important in what follows.

A key feature of the scaling in \eqref{eq: fast-slow system} is that the variance of the invariant measure for $X^\tau$ does not depend on $\tau$: indeed, for the OU process $dX_t^\tau = -(1/\tau) (X_t ^\tau- Q)dt + (1/\sqrt{\tau }) dW_t$ the stationary law is Gaussian with mean $Q$ and variance $1/2.$ Thus $\tau$ controls the \emph{correlation time} of $X^\tau$, not its variance.

At this point, we would like to deduce a closed equation for the slow process $T_t^{\tau,\lambda}$.
\begin{lemma}
For any $t>0$, it holds
$$
   \lim_{\tau \to 0} \mathbb{E} \abs{\frac{1}{\sqrt{\tau}} \int_0^t X^\tau_s ds - \frac{Qt}{\sqrt{\tau}} - W_t }^2  =   0.
$$
\label{lemma: approssimazione per tau piccolo Q}
\end{lemma}
\begin{proof}
For simplicity, we drop the dependence on $\tau$ in $X^\tau$. The fast process $X_t$ can be written explicitly as
    \begin{equation*}
        \begin{split}
                X_t &= x_0 e^{-t/\tau}+ \frac{1}{\tau} \int_0^t Qe^{-(t-s)/\tau} ds + \frac{1}{\sqrt{\tau}} \int_0^t e^{-(t-s)/\tau }dW_s \\
                & = x_0 e^{-t/\tau} + Q \left( 1-e^{-t/\tau} \right) + \frac{1}{\sqrt{\tau}} \int_0^t e^{-(t-s)/\tau }dW_s \\
                & = \left( x_0 - Q \right)e^{-t/\tau} + Q + \frac{1}{\sqrt{\tau}} \int_0^t e^{-(t-s)/\tau } dW_s.
        \end{split}
    \end{equation*}
    Thus, integrating in time, we get
    \begin{equation*}
        \begin{split}
            \int_0^t X_s ds &= \int_0^t   \left[\left( x_0 - Q \right)e^{-s/\tau} + Q + \frac{1}{\sqrt{\tau}} \int_0^s e^{-(s-r)/\tau } dW_r \right] ds \\
            & =  \tau \left( 1-e^{-t/\tau } \right)\left( x_0 - Q \right) + t Q + \frac{1}{\sqrt{\tau}} \int_0^t \int_0^s e^{-(s-r)/\tau } dW_r  ds.
        \end{split}
    \end{equation*}
    The last term on the right-hand side (RHS) can be computed applying Fubini-Tonelli for stochastic integrals
    \begin{equation*}
        \begin{split}
             \frac{1}{\sqrt{\tau}} \int_0^t \int_0^s e^{-(s-r)/\tau } dW_r  ds &=  \frac{1}{\sqrt{\tau}} \int_0^t \int_r^t e^{-(s-r)/\tau } ds dW_r   = \sqrt{\tau}  \int_0^t \left(1-e^{-(t-r)/\tau}  \right) dW_r \\
            & = \sqrt{ \tau} \left( W_t - \int_0^t e^{-(t-r)/\tau} dW_r \right).
        \end{split}
    \end{equation*}
    From the previous computations, we deduce
    \begin{equation*}
        \begin{split}
             \mathbb{E} \abs{\frac{1}{\sqrt{\tau}} \int_0^t X_s ds - \frac{Qt}{\sqrt{\tau}} - W_t }^2 &= \mathbb{E} \abs{\sqrt{\tau} \left(1 - e^{-t/\tau} \right) \left(x_0 - Q\right) - \int_0^t e^{-(t-r)/\tau} dW_r }^2 \\
             & \leq 2\tau \left(1 - e^{-t/\tau} \right)^2 \left(x_0 - Q\right)^2 + 2 \mathbb{E} \abs{\int_0^t e^{-(t-r)/\tau} dW_r}^2 \\
             &= 2\tau \left(1 - e^{-t/\tau} \right)^2 \left(x_0 - Q\right)^2 + 2   \int_0^t e^{-2(t-r)/\tau} dr \\
             & = 2\tau \left(1 - e^{-t/\tau} \right)^2 \left(x_0 - Q\right)^2 +  \tau \left( 1- e^{-2t/\tau} \right)
        \end{split}
    \end{equation*}
    where the first inequality follows from the Jensen's inequality, and the following inequality by It\^{o}-isometry. Taking the limits in $\tau$, we conclude the proof.
    
\end{proof}
Define the centred, integrated fast fluctuations
$$
W_t^\tau := \frac{1}{\sqrt{\tau}}\int_0^t (X_s^\tau-Q)\,ds .
$$
Since $X^\tau$ has continuous paths, $W^\tau$ is $C^1$ and
$$
dW_t^\tau = \frac{1}{\sqrt{\tau}}(X_t^\tau-Q)\,dt .
$$
Lemma~\ref{lemma: approssimazione per tau piccolo Q} yields $W_t^\tau \to W_t$ in $L^2$ for each fixed $t$.
Using this identity, the temperature equation in \eqref{eq: fast-slow system} can be rewritten as the random ODE
\begin{equation}
dT_t ^{\tau,\lambda} =
\left(Q\beta(T_t^{\tau,\lambda})+\lambda-\mathcal{R}_e(T^{\tau,\lambda} _t)\right)\,dt
+\sqrt{\tau}\,\beta(T_t^{\tau,\lambda})\,dW_t^\tau ,
\label{eq: temperature equation integrated by V tau}
\end{equation}
where the last integral is understood pathwise (Riemann-Stieltjes), since $W^\tau$ is $C^1$. We now invoke the Wong--Zakai principle (see Section~\ref{subsec: preliminaries}).
Since \eqref{eq: temperature equation integrated by V tau} is driven by the smooth process $W^\tau$, replacing $W^\tau$
by its Brownian limit $W$ leads to a Stratonovich interpretation of the limiting stochastic integral.
Accordingly, for $\tau\ll 1$ we approximate $T^{\tau,\lambda}$ by the solution $\widetilde T ^{\tau,\lambda}$ of
\begin{equation}
\label{eq: Stratonovich EBM sec 1}
d \widetilde{T}_t ^{\tau,\lambda}
= \left(Q\beta(\widetilde{T}_t ^{\tau,\lambda})+\lambda-\mathcal{R}_e(\widetilde T_t ^{\tau,\lambda})\right)\,dt
+ \sqrt{\tau}\,\beta(\widetilde{T}_t ^{\tau,\lambda})\circ dW_t.
\end{equation}
In It\^{o} form, \eqref{eq: Stratonovich EBM sec 1} becomes 
\begin{equation}
d\widetilde{T}_t ^{\tau,\lambda}
= \left(Q\beta(\widetilde{T}_t ^{\tau,\lambda})+\lambda-\mathcal{R}_e(\widetilde T_t^{\tau,\lambda})
+\frac{\tau}{2}\beta(\widetilde{T}_t ^{\tau,\lambda})\beta'(\widetilde{T}_t ^{\tau,\lambda})\right)dt
+ \sqrt{\tau}\,\beta(\widetilde{T}_t ^{\tau,\lambda}) dW_t.
\label{eq: closed EBM ito}
\end{equation}
We stress that we are not taking a literal limit $\tau\to0$ from \eqref{eq: temperature equation integrated by V tau} to \eqref{eq: Stratonovich EBM sec 1}. Doing so would indeed make the diffusion coefficient vanish. Rather, $\tau$ is the (small but positive) ratio between the fast and slow time scales, and the passage from
\eqref{eq: temperature equation integrated by V tau} to \eqref{eq: Stratonovich EBM sec 1} is a
\emph{model reduction step}: for $\tau\ll1$, the forcing generated by the fast OU process is
approximated by a white-noise forcing, with Stratonovich interpretation suggested by Wong-Zakai.

We study fluctuations near the stable equilibrium $T^{*,\lambda}$, so the \emph{temperature anomaly} $$
Y_t ^{\tau,\lambda}: = \widetilde T_t^{\tau,\lambda}- T^{*, \lambda}
$$ stays small and it is natural to linearise $\beta(T)$ around $T^{*,\lambda}$. The It\^{o}-Stratonovich correction is an $O(\tau)$ drift term, and since $ \tau \approx 1/365$ (with $\beta, \beta '$ bounded) it is much smaller than the remaining  deterministic drift  modelling the radiation balance $\mathcal{R}_a- \mathcal{R}_e $. So we can neglect it. Thus, we define $T^{*,\lambda}$ as the stable equilibrium temperature $T^* = T^{*,\lambda}$ of the deterministic stationary equation associated to \eqref{eq: closed EBM ito}. We stress that it is an assumption of our analysis that $T^{*,\lambda}$ is a stable equilibrium point for the deterministic EBM \eqref{eq: EBM preliminaries}. In other words, $T^{*,\lambda}$ is solution of
\begin{equation}
    Q \beta(T^{*,\lambda}) + \lambda - \mathcal{R}_e(T^{*,\lambda})  = 0,
    \label{eq: equation equilibrium temperature}
\end{equation}
and it is (locally asymptotically) stable for the deterministic dynamics, i.e. the linearisation of
$\frac{d T}{dt} = Q\beta(T)+\lambda-\mathcal R_e(T)$ at $T^{*,\lambda}$ has negative slope
$$
\frac{d}{dT}(Q\beta(T)+\lambda-\mathcal R_e(T))|_{T=T^{*,\lambda}}
=Q\beta'(T^{*,\lambda})-\mathcal R_e'(T^{*,\lambda})<0,
$$
equivalently 
\begin{equation}
b^\lambda=\mathcal R_e'(T^{*,\lambda})-Q\beta'(T^{*,\lambda})>0
\label{eq: b positive section 1}
\end{equation}
Since $\beta$ is piecewise linear by \eqref{eq: piecewise co-albedo}, for $T$ in a neighbourhood of $T^{*,\lambda}$ that does not cross $T_l$ or $T_u$ we have the affine representation
\begin{equation}
    \beta (T) = \beta (T^{*,\lambda })+ \beta'(T^{*,\lambda}) (T - T^{*,\lambda}),
    \label{eq: linearisation co-albedo}
\end{equation}
with $\beta'(T^{*,\lambda})=0$ outside $(T_l,T_u)$ and $\beta'(T^{*,\lambda})=(\beta_{\max}-\beta_{\min})/(T_u-T_l)$ inside. Thus, recalling that $\mathcal{R}_e (T) = r_0 +r_1T$ is already linear, we get, applying \eqref{eq: linearisation co-albedo} in \eqref{eq: closed EBM ito} and using \eqref{eq: equation equilibrium temperature} to cancel the constant drift,
\begin{equation}
    dY_t ^{\tau,\lambda} = \left( Q  \beta'(T^{*,\lambda}) -r_1\right) Y_t ^{\tau,\lambda} dt + \sqrt{\tau }\left(  \beta(T^{*,\lambda}) + \beta'(T^{*,\lambda}) Y_t ^{\tau,\lambda}\right) dW_t, \qquad Y_0 = y_0,
    \label{eq: linear EBM}
\end{equation}
where $y_0$ is some deterministic initial condition. Introducing the shorthands
$$
b^\lambda := -Q  \beta'(T^{*,\lambda}) +r_1, \qquad \sigma_0^\lambda := \beta(T^{*,\lambda}), \qquad \sigma_1^\lambda := \beta'(T^{*,\lambda}),
$$
then the linear EBM \eqref{eq: linear EBM} assumes the compact form
$$
d Y_t ^{\tau,\lambda} = -b^\lambda Y_t ^{\tau,\lambda} dt + \sqrt{\tau  }(\sigma_0^\lambda + \sigma_1^\lambda Y_t ^{\tau,\lambda}) dW_t, \qquad Y_0^{\tau,\lambda} = y_0.
$$

\begin{remark}[On the stability condition \eqref{eq: b positive section 1}]
Condition \eqref{eq: b positive section 1} is the stability requirement for the deterministic EBM at
$T^{*,\lambda}$. It is automatic outside the ice-sensitive range (i.e. $T^{*,\lambda} < T_l$ or $T^{*,\lambda} > T_u$) because $\beta'(T^{*,\lambda})=0$ and thus
$b^\lambda=r_1>0$. In the ice-sensitive range (i.e. $T^{*,\lambda} \in (T_l,T_u)$) the sign of $b^\lambda=r_1-Q\beta'(T^{*,\lambda})$
depends on the magnitude of $Q$. For an Arctic winter mean, we can reasonably consider $Q\approx 0$; so stability is
expected. On the other hand, for summer or annual-mean forcing the condition can fail. Since our stationary-variance
analysis requires a stable linearisation, we treat \eqref{eq: b positive section 1} as an explicit modelling
assumption delimiting the regime of validity. A more realistic approach would introduce seasonal forcing
$Q=Q(t)$, yielding a non-autonomous model; we leave this for future work.
\end{remark}

\subsection{Stationary variance and monotone dependence on radiative forcing}
\label{subsec: stationary variance and monotone dependence on CO2}
In this section we prove that the stationary variance associated to the solution of \eqref{eq: linear EBM} increases monotonically with respect to $\lambda$. This means that, as the $\text{CO}_2$ concentration in the atmosphere increases, the same holds for the variance of the climate, i.e. the stationary temperature anomaly.

We start by observing that, since $t \mapsto \int_0^t( \sigma_0^\lambda + \sigma_1^\lambda  Y_s ^{\tau,\lambda}) \, dW_s$ is a martingale, then the expected value of the process $Y_t ^{\tau,\lambda}$ satisfies 
\begin{equation}
\mathbb{E}[Y^{\tau,\lambda}_t]= y_0  -b^\lambda \int_0^t \mathbb{E}[Y^{\tau,\lambda}_s]\, ds.
\label{eq: ODE mean temperature anomaly}
\end{equation}
This can be seen by writing \eqref{eq: linear EBM} in the integral form and then considering the expectation. The solution of \eqref{eq: ODE mean temperature anomaly} is
$$
\mathbb{E}[ Y_t ^{\tau,\lambda}] = e^{-b^\lambda t}y_0, \qquad t \geq 0.
$$
Thus the limit for large time $t$ of the expected value of $Y_t^{\tau,\lambda}$ exists and it is independent of $y_0$ if and only if $ b^\lambda >0.$ In that case, we have
$$
\lim \limits_{t \to +\infty} \mathbb{E}[ Y_t^{\tau,\lambda}   ] = 0.
$$
This motivates once more why the assumption $b^\lambda >0$ from \eqref{eq: b positive section 1} is natural in this framework. Further, to simplify the computation, we also assume $y_0 = 0$, which gives $\mathbb{E}[Y_t^{\tau,\lambda}] = 0$ for any $ t\geq 0.$ The following result identifies the value of the stationary variance associated to the equation \eqref{eq: linear EBM}.
\begin{proposition}
\label{prop: stationary variance}
    Assume 
    $$
    b^\lambda >0,  \qquad  2b^\lambda - \tau (\sigma_1^\lambda )^2 >0.
    $$
    Then the following limit exists
    $$
    \lim \limits_{ t \to +\infty } \mathbb{E}[ \vert Y_t^{\tau,\lambda}  \vert ^2]=  \frac{\tau (\sigma_0^\lambda)^2}{2b^{\lambda}- \tau (\sigma_1^\lambda)^2}.
    $$
\end{proposition}
\begin{proof}
    We drop the superscript $\tau,\lambda $ for readability. We assume for simplicity $y_0 = 0$ and use the notation $V(t) = \mathbb{E}[ \vert Y_t \vert ^2]$. Thus $Y_t$ is a centered process. The It\^{o} formula gives
    \begin{equation}
        \begin{split}
            d(Y_t)^2 &= 2Y_t dY_t + d [Y_\cdot]_t\\ &= 2 Y_t \left(  - b Y_t dt + \sqrt{\tau}(\sigma_0 + \sigma_1Y_t) dW_t\right) + \tau (\sigma_0 + \sigma_1 Y_t)^2 dt\\
            &= \left[-(2b - \tau \sigma_1^2) Y_t^2  + 2 \tau \sigma_0 \sigma_1 Y_t  + \tau \sigma_0^2\right] dt + dM_t,
        \end{split}
    \end{equation}
    where $[Y_\cdot]_t$ denotes the quadratic variation of the process $Y_t$ and $M_t $ is a martingale. Considering the expected value in the previous equality, and using that $\mathbb{E}[Y_t] = 0 $, yields
    $$
    \frac{d}{dt} V(t) =  -(2b - \tau \sigma_1^2) V(t) + \tau \sigma_0^2.
$$
Since $V(0)=0,$ its solution is
$$
V(t) = \tau \sigma_0^2 \int_0^t e^{-(2b - \tau \sigma_1^2)(t-s)} ds = \frac{\tau \sigma_0^2}{2b - \tau \sigma_1^2} \left( 1 - e^{-(2b - \tau \sigma_1^2)t} \right).
$$
Hence
$$
\lim \limits_{t \to +\infty } V(t) = \frac{\tau \sigma_0^2}{ 2b - \tau \sigma_1^2},
$$
provided that $2b - \tau \sigma_1^2 >0$.
\end{proof}
We are now ready to prove the main result of this section, which deals with the increase of the stationary variance under an increase of the radiative forcing, modelled by $\lambda$.
\begin{proposition} \label{prop: variance derivative}
Assume 
    $$
    b^\lambda >0,  \qquad  2b^\lambda - \tau (\sigma_1^\lambda )^2 >0, \qquad T^{*,\lambda} \in (T_l, T_u).
    $$
    Then the stationary variance is monotone increasing. In other words, setting
$$
\lambda \mapsto \mathrm{Var}_\infty (\lambda):= \lim \limits_{ t \to +\infty  } \mathbb{E}[ \vert Y_t^{\tau,\lambda} \vert ^2] = \frac{\tau (\sigma_0^\lambda )^2}{ 2 b^\lambda - \tau (\sigma_1^\lambda)^2} = \frac{\tau \beta(T^{*,\lambda})^2 }{2 (r_1 - Q \beta'(T^{*,\lambda})) - \tau \beta'( T^{*,\lambda})^2},
$$
we have $\mathrm{Var}_\infty (\lambda) ' >0.$
\end{proposition}
\begin{proof}
Assume for simplicity $y_0 = 0$. Recall that \(T^{*,\lambda}\) is defined by
$$
Q\beta(T^{*,\lambda})+\lambda-\mathcal R_e(T^{*,\lambda})=0,
$$
and that the linearised coefficients are
$$
b^\lambda = r_1 - Q\beta'(T^{*,\lambda}), 
\qquad 
\sigma_0^\lambda=\beta(T^{*,\lambda}), 
\qquad 
\sigma_1^\lambda=\beta'(T^{*,\lambda}).
$$
By Proposition \ref{prop: stationary variance}, the stationary variance of the temperature anomaly \(Y_t^{\tau,\lambda} =\widetilde T_t^{\tau,\lambda}-T^{*,\lambda}\) is
$$
\mathrm{Var}_\infty(\lambda)
=\lim_{t\to\infty}\mathbb E|Y_t^{\tau,\lambda} |^2
=\frac{\tau(\sigma_0^\lambda)^2}{2b^\lambda-\tau(\sigma_1^\lambda)^2}
=
\frac{\tau\,\beta(T^{*,\lambda})^2}{\,2\bigl(r_1-Q\beta'(T^{*,\lambda})\bigr)-\tau\,\beta'(T^{*,\lambda})^2\,}.
$$

\textit{Step 1: increasing $\lambda$ increases the equilibrium temperature $T^{*,\lambda}$.}
Differentiating implicitly the equilibrium equation in \(\lambda\) gives
$$
\left(Q\beta'(T^{*,\lambda})-\mathcal{ R}_e'(T^{*,\lambda})\right) \frac{dT^{*,\lambda}}{d\lambda}+1=0.
$$
Since $\mathcal R_e'(T)\equiv r_1 $ and $b^\lambda=r_1-Q\beta'(T^{*,\lambda})$, this becomes
$$
\frac{d T^{*,\lambda }}{d  \lambda}=\frac{1}{b^\lambda}>0.
$$

\textit{Step 2: monotone increase in stationary variance.}
For simplicity, set
$$
N(\lambda):= \tau \beta(T^{*,\lambda})^2,  \qquad D(\lambda) := 2\bigl(r_1-Q\beta'(T^{*,\lambda})\bigr)-\tau\,\beta'(T^{*,\lambda})^2
$$
so that
$$
\mathrm{Var}_\infty (\lambda) = \frac{N(\lambda )}{D(\lambda)}.
$$
In Step $1$ we have proved that $\lambda \mapsto T^{*,\lambda}$ is monotone increasing. Further, the co-albedo $\beta$ is increasing by the standard modelling of EBMs. Hence
\begin{equation}
    N'(\lambda) = 2 \tau \beta(T^{*,\lambda}) \beta'(T^{*,\lambda }) \frac{d T^{*,\lambda }}{d  \lambda} >0.
    \label{eq: increasing numerator}
\end{equation}
Note that $\beta'(T^{*,\lambda})$ is non-zero (and positive), since $T^{*,\lambda} \in (T_l, T_u)$, thanks to \eqref{eq: piecewise co-albedo}. Regarding $D(\lambda),$ since $\beta$ is piece-wise linear, then $\beta'' = 0.$ Hence
\begin{equation}
D'(\lambda) =- 2 \beta '' (T^{*,\lambda})( Q+ \tau \beta ' (T^{*,\lambda})) \frac{d T^{*,\lambda}}{d\lambda}=0.
    \label{eq: decreasing denominator}
\end{equation}

Combining \eqref{eq: increasing numerator} and \eqref{eq: decreasing denominator}, we conclude that $$
\mathrm{Var}'_\infty (\lambda) = \frac{N'(\lambda ) D(\lambda ) - N(\lambda) D'(\lambda)}{D(\lambda)^2} = \frac{N'(\lambda ) D(\lambda )}{D(\lambda)^2}>0,
$$
since $D(\lambda)>0$ thanks to the assumption $2b^\lambda - \tau (\sigma_1^\lambda)^2 >0$.
\end{proof}
The previous result indicates, in climate terms, that stronger $\mathrm{CO}_2$ forcing not only warms the equilibrium states, but also that, via the ice-albedo feedback, amplifies the temperature anomaly around the equilibrium state. Indeed, the monotone increase of the stationary variance can be read directly from the explicit formula
$$
\mathrm{Var}_\infty(\lambda)
=\frac{\tau\,\beta(T^{*,\lambda})^2}{\,2\bigl(r_1-Q\beta'(T^{*,\lambda})\bigr)-\tau\,\beta'(T^{*,\lambda})^2\,}.
$$
By Step 1 in the previous proof, increasing $\lambda$ increases the equilibrium temperature $T^{*,\lambda}$. If $T^{*,\lambda}\in (T_l,T_u)$, then $\beta$ is strictly increasing, hence the numerator
$\tau\,\beta(T^{*,\lambda})^2$ increases with $\lambda$. Moreover, since $\beta$ is affine on $(T_l,T_u)$, we have $\beta'(T^{*,\lambda})\equiv (\beta_{\max}-\beta_{\min})/(T_u-T_l)$,
so the denominator $2(r_1-Q\beta'(T^{*,\lambda}))-\tau\,\beta'(T^{*,\lambda})^2$ is constant (in this regime). Therefore $\mathrm{Var}_\infty(\lambda)$ increases with $\lambda$. In climate terms, stronger $\mathrm{CO}_2$ forcing not only warms the equilibrium state, but also amplifies fluctuations around it via the temperature dependence of the co-albedo (ice-albedo feedback).

\section{Spatial anomaly model and finite-dimensional approximation} 
\label{sec: two}
In this section we extend the zero-dimensional stochastic energy balance model \eqref{eq: closed EBM ito} from Section \ref{sec: one} to a space-dependent version, appropriate for a regional model (the Arctic). This is done by introducing horizontal heat transport. We first formulate in Section \ref{subsec: spde dirichlet} an energy-balance SPDE with \emph{multiplicative} noise, with Dirichlet non-homogeneous boundary data representing the influence of adjacent areas. In Section \ref{subsec: fd anomaly} we then construct a finite-difference semi-discretisation, obtaining a finite-dimensional SDE for the temperature anomaly. In Section \ref{subsec: Stationary covariance for the semi-discrete anomaly SDE} we derive the closed evolution equation for the covariance matrix of the SDE and identify its stationary limit. Finally, in Section \ref{subsec: monotone covariance} we prove that the stationary covariance increases  (entry-wise) with the radiative forcing parameter $\lambda$, given natural assumptions on the covariance of the noise and the deterministic dynamics.

\subsection{Spatial EBM with Dirichlet data and anomaly formulation}
\label{subsec: spde dirichlet} 
Let $\mathcal D\subset\mathbb{R}^2$ be a bounded domain with sufficiently regular boundary $\partial  \mathcal D$.
We consider a Hasselmann-type stochastic energy balance model with diffusion, i.e. the spatially dependent extension of equation \eqref{eq: closed EBM ito}. The (near-surface)
temperature field $T^{\tau,\lambda}(t,x)$, depending on the time-scale parameter $\tau$ and the radiative forcing $\lambda$, satisfies non-homogeneous Dirichlet boundary conditions
\begin{equation}
T^{\tau,\lambda}(t,\xi)=\theta(\xi), \qquad t \geq 0, \ \xi\in\partial  \mathcal D,
\label{eq: dirichlet T}
\end{equation}
where $\theta$ is a prescribed boundary profile; in the following, we assume $\theta $ to be independent of the radiative forcing parameter $\lambda$, but our analysis extends to the case of $\theta = \theta^\lambda$.
On $\mathcal D$ we postulate the It\^o SPDE
\begin{equation}
\left \lbrace 
\begin{aligned}
dT_t ^{\tau,\lambda} &= \left(\Delta T_t ^{\tau,\lambda}+ Q\,\beta(T_t^{\tau,\lambda})+\lambda-\mathcal R_e(T_t ^{\tau,\lambda})\right)dt
+ \sqrt{\tau}\,\beta(T_t ^{\tau,\lambda})\,dW_t, \\
T^{\tau,\lambda}|_{\partial \mathcal{D}} &= \theta\\
T^{\tau,\lambda}|_{t = 0} & = T_0 
\end{aligned}
\label{eq: SPDE T}
\right.
\end{equation}
where $ Q= Q(x)$ denotes the given insolation function, $\beta$ is the
co-albedo, $\lambda$ is the radiative forcing parameter, $W_t$ is a Wiener process on $L^2(\mathcal{D})$, and $T_0$ is the initial condition. We do not specify the details of these objects here, as we use them only heuristically; the interested reader may consult \cite{DPZ}.

Assume, for simplicity, that the Dirichlet datum $\theta$ is time-independent.
We define the reference profile $T^{*,\lambda}$ as a (stable) stationary solution of the deterministic
problem with the same boundary values
\begin{equation}
\begin{cases}
\Delta T^{*,\lambda}(x) + Q(x)\,\beta(T^{*,\lambda}(x)) + \lambda - \mathcal{R}_e(T^{*,\lambda}(x))=0,
& x\in \mathcal  D,\\
T^{*,\lambda}(\xi)=\theta(\xi), & \xi\in\partial \mathcal D.
\end{cases}
\label{eq: elliptic equilibrium dirichlet}
\end{equation}
The temperature anomaly is then defined by
\[
Y_t^{\tau,\lambda}(x) :=T_t^{\tau,\lambda}(x)-T^{*,\lambda}(x).
\]
By construction, \eqref{eq: dirichlet T} and \eqref{eq: elliptic equilibrium dirichlet} yield
\begin{equation}
Y_t^{\tau,\lambda}(\xi)=0,\qquad \xi \in \partial \mathcal{D},
\label{eq: dirichlet Y homogeneous}
\end{equation}
i.e. the anomaly satisfies homogeneous Dirichlet boundary conditions. Subtracting \eqref{eq: elliptic equilibrium dirichlet} from \eqref{eq: SPDE T} gives the anomaly SPDE
\begin{equation}
\label{eq: SPDE anomaly exact}
dY_t^{\tau,\lambda}=
\left(\Delta Y_t^{\tau,\lambda} + Q(\beta(T^{*,\lambda}+Y_t^{\tau,\lambda})- \beta(T^{*,\lambda}))-(\mathcal{R}_e(T^{*,\lambda}+Y_t^{\tau,\lambda})-\mathcal{R}_e(T^{*,\lambda}))\right)dt
+\sqrt{\tau}\,\beta(T^{*,\lambda}+Y_t^{\tau,\lambda})\,dW_t,
\end{equation}
with homogeneous Dirichlet boundary condition \eqref{eq: dirichlet Y homogeneous}. As in the zero-dimensional case, if $y $ is small, it holds the relation
$$
\beta(T^{*,\lambda}+y) =  \beta(T^{*,\lambda})+\beta'(T^{*,\lambda})\,y.
$$
Since in our model $\mathcal{R}_e$ is affine (e.g.\ $\mathcal{R}_e(T)=r_0+r_1T$), we obtain the linear multiplicative SPDE
\begin{equation}
\label{eq: SPDE anomaly linear}
\left \lbrace
\begin{aligned}
    dY_t ^{\tau,\lambda}
&=
\left(\Delta Y_t ^{\tau,\lambda} - a^\lambda Y_t ^{\tau,\lambda}\right)dt
+\sqrt{\tau}\left(\beta(T^{*,\lambda})+\beta'(T^{*,\lambda})\,Y_t ^{\tau,\lambda}\right)dW_t,\\
Y^{\tau,\lambda}|_{\partial \mathcal{D} } & = 0, \\
Y^{\tau,\lambda}|_{t = 0} &= T_0 - T^{*,\lambda},
\end{aligned}
\right.
\end{equation}
where
$$
a^\lambda(x):=- Q(x)\,\beta'(T^{*,\lambda}(x))+\mathcal{R}_e'(T^{*,\lambda}(x)).
$$
Equation \eqref{eq: SPDE anomaly linear} is complemented with homogeneous Dirichlet boundary conditions
\eqref{eq: dirichlet Y homogeneous}. We are not interested in the well-posedness of \eqref{eq: SPDE anomaly linear} here; the SPDE is a formal starting point for the semi-discrete covariance analysis.

\subsection{Finite-difference semi-discretisation}
\label{subsec: fd anomaly}

We now derive a finite-difference approximation of \eqref{eq: SPDE anomaly linear} on a two-dimensional mesh.
For simplicity we present the construction on a rectangular grid; the resulting semi-discrete SDE has the same
structure on general meshes, with $A_\Delta$ replaced by the corresponding stiffness matrix. We refer to \cite{Quarteroni,Leveque} for detailed treatments of finite difference methods for parabolic PDEs.

To describe the numerical discretisation, assume for simplicity $\mathcal D=(0,L_x)\times(0,L_y)$. Consider the uniform grid
$x_i=i h_x$ for $i=0,\dots,N_x$ and $y_j=j h_y$ for $j=0,\dots,N_y$.
We denote by
\[
\mathcal I:=\{1,\dots,N_x-1\}\times\{1,\dots,N_y-1\}
\]
the set of interior index pairs, and by
\[
z_{i,j}:=(x_i,y_j)\in\mathcal D,\qquad (i,j)\in\mathcal I,
\]
the corresponding interior grid points. We enumerate the interior nodes by a single index
$m=1,\dots,d$ (with $d=(N_x-1)(N_y-1)$) via the bijection
\[
m=m(i,j):=(j-1)(N_x-1)+i,
\qquad (i,j)\in\mathcal I,
\]
and write $z_m:=z_{i,j}$ whenever $m=m(i,j)$. Define the nodal unknowns $Y_{i,j} ^{\tau,\lambda}(t):=Y_t^{\tau,\lambda}(z_{i,j})$ for $(i,j)\in\mathcal I$.
Using the bijection $m=m(i,j)$, we collect these interior nodal values into the vector
$Y_t ^{\tau,\lambda}\in\mathbb R^d$, with $d=(N_x-1)(N_y-1)$, by setting $(Y_t ^{\tau,\lambda})_m := Y_{i,j}^{\tau,\lambda}(t)$ whenever $m=m(i,j)$. The five-point stencil for the Laplacian at an interior node reads
$$
(\Delta_h Y^{\tau,\lambda})_{i,j}
=
\frac{Y_{i+1,j}^{\tau,\lambda}-2Y_{i,j}^{\tau,\lambda}+Y_{i-1,j}^{\tau,\lambda}}{h_x^2}
+
\frac{Y_{i,j+1}^{\tau,\lambda}-2Y_{i,j}^{\tau,\lambda}+Y_{i,j-1}^{\tau,\lambda}}{h_y^2},
$$
see \cite[Section 3.2]{Leveque} for more details on the discretisation of the Laplace operator. Since $Y^{\tau,\lambda}$ satisfies homogeneous Dirichlet boundary conditions \eqref{eq: dirichlet Y homogeneous}, boundary
values are zero and no boundary forcing appears. In matrix form,
\begin{equation}
\Delta_h Y_t^{\tau,\lambda} = A_\Delta Y_t^{\tau,\lambda},
\label{eq: discrete laplacian homogeneous}
\end{equation}
where
\begin{equation}
A_\Delta =
\frac{1}{h_x^2}\left(I_{N_y-1}\otimes A_x\right)
+
\frac{1}{h_y^2}\left(A_y\otimes I_{N_x-1}\right),
\label{eq: Adelta kron sum}
\end{equation}
and
$
A_x=\mathrm{tridiag}(1,-2,1)\in\mathbb R^{(N_x-1)\times(N_x-1)},$ $
A_y=\mathrm{tridiag}(1,-2,1)\in\mathbb R^{(N_y-1)\times(N_y-1)}.$ Define the sampled coefficients on interior nodes
$$
b_m^\lambda := a^\lambda(z_m),\qquad
d_m^\lambda := \beta'(T^{*,\lambda}(z_m)),\qquad
f_m^\lambda := \beta(T^{*,\lambda}(z_m)),\qquad m=1,\dots,d.
$$
Set
$$
B^\lambda := \mathrm{diag}(b^\lambda)  \in \mathbb{R}^{d\times d}, \qquad D^\lambda := \mathrm{diag}(d^\lambda)\in \mathbb{R}^{d\times d},
\qquad f^\lambda :=     (f_i^\lambda)_{i=1}^d   \in \mathbb{R}^d.
$$
To discretise the spatially correlated noise, let $C\in\mathbb{R}^{d\times d}$ be the discrete covariance
matrix induced by the covariance operator of $W$ on the grid, and fix a factorisation $C=LL^T$.
Let $(W_t)_t$ be a $d$-dimensional Brownian motion. Then the semi-discrete approximation of
\eqref{eq: SPDE anomaly linear} reads
\begin{equation}
dY_t^{\tau,\lambda}= (A_\Delta- B^\lambda) Y_t^{\tau,\lambda} \,dt
+\sqrt{\tau} \mathrm{diag}\left(D^\lambda Y_t^{\tau,\lambda} + f^\lambda\right) L dW_t,
\qquad Y^{\tau,\lambda}|_{t=0}=y_0.
\label{eq: SDE anomaly fd}
\end{equation}

We conclude this section by introducing a scalar proxy for the fluctuation level of the semi-discrete anomaly process \eqref{eq: SDE anomaly fd}, given by the trace of its covariance matrix. From now on, fix $\tau>0$ and write $Y^\lambda := Y^{\tau,\lambda}$.

\begin{definition}[Spatial Variance]
    Let $(Y_t^{\lambda})_t $ be the solution of the SDE \eqref{eq: SDE anomaly fd} and let $\Gamma_t ^{\lambda} := \mathbb{E}[ (Y_t^{\lambda}- \mathbb{E}[Y^{\lambda}_t]) (Y_t^{\lambda} - \mathbb{E}[Y^\lambda_t])^T]$ be its covariance matrix. The \emph{spatial variance} at time $t$ is
    $$
    \mathrm{Var}_{\mathrm{sp}} (t;\lambda ) := \mathrm{Tr} (\Gamma_t^\lambda),
    $$
    and, whenever the limit exists, the \emph{stationary spatial variance} is
    $$
    \mathrm{Var}_{\mathrm{sp},\infty }(\lambda) := \mathrm{Tr} (\Gamma_\infty ^\lambda), \qquad \Gamma_\infty ^\lambda := \lim_{t \to \infty } \Gamma_t^\lambda.
    $$
    \label{def: spatial variance}
\end{definition}
In the discussion of spatial variance below we assume $y_0=0$. In the next subsection we briefly allow general $y_0$ to describe the mean, and then we set $y_0=0$ again to obtain an autonomous covariance equation. Thus, with the assumption $y_0 = 0$, from the fact that the drift is linear and the stochastic integral is a martingale,
it follows that $\mathbb{E}[Y_t^\lambda]=0$ for all $t\geq 0$. In this case,
$$
\mathrm{Var}_{\mathrm{sp}}(t;\lambda)=\mathrm{Tr}(\Gamma_t^\lambda)
=\sum_{i=1}^d \mathrm{Var}(Y_{t,i}^\lambda)
=\mathbb{E}\left[\sum_{i=1}^d |Y_{t,i}^\lambda|^2\right].
$$
So $\mathrm{Var}_{\mathrm{sp}}(t; \lambda)$ is the sum of the local fluctuation intensities over the domain, and
$\mathrm{Var}_{\mathrm{sp},\infty}(\lambda)$ measures this total variance in the stationary regime.

\begin{remark}[Stationary spatial variance and extreme weather events] 
It is natural to focus on the \emph{stationary} spatial variance
$\mathrm{Var}_{\mathrm{sp},\infty}(\lambda)=\mathrm{Tr}(\Gamma_\infty^\lambda)$, since it measures the fluctuation level of the
temperature anomaly in the long-time regime corresponding to the \emph{climate configuration} determined by the
radiative forcing $\lambda$ (after transients have died out). In particular, comparing
$\mathrm{Var}_{\mathrm{sp},\infty}(\lambda_1)$ and $\mathrm{Var}_{\mathrm{sp},\infty} (\lambda_2)$ quantifies how the typical
amplitude of persistent spatial fluctuations changes when the forcing level (e.g.\ $\mathrm{CO}_2$) is modified.

Although extreme events are not determined by second moments alone, $\mathrm{Var}_{\mathrm{sp}}(t;\lambda)$ provides a
useful second-moment control on the probability of observing a large anomaly somewhere in the domain. Assume for simplicity $y_0=0$, so that $\mathbb{E}[Y_t^\lambda]=0$ for all $t\geq 0$.  For any
$\vartheta>0$, by Markov's inequality
$$
\mathbb{P}\left( \max_{ 1 \leq i \leq d}  \vert Y_{t,i}^\lambda  \vert  \geq \vartheta \right) \leq \frac{1}{\vartheta ^2} \mathbb{E}\left[ \max_{ 1 \leq i \leq d} \vert Y_{t,i}^\lambda \vert ^2\right] \leq \frac{1}{\vartheta^2} \mathbb{E}\left[ \sum_{ i = 1}^d \vert  Y_{t,i}^\lambda \vert^2 \right] = \frac{\mathrm{Var}_{\mathrm{sp}}(t;\lambda)}{\vartheta ^2}.
$$
and the same estimate holds in the stationary regime with $\mathrm{Var}_{\mathrm{sp},\infty}(\lambda)$.
\end{remark}

\subsection{Stationary covariance for the semi-discrete anomaly SDE}
\label{subsec: Stationary covariance for the semi-discrete anomaly SDE}
Recall from Definition~\ref{def: spatial variance} that our scalar proxy for the fluctuation level is the spatial variance $\mathrm{Var}_{\mathrm{sp}}(t;\lambda)=\Tr(\Gamma_t^\lambda)$ and its stationary limit $\mathrm{Var}_{\mathrm{sp},\infty}(\lambda)=\Tr(\Gamma_\infty^\lambda)$ whenever it exists. Thus, establishing existence and monotonicity properties for $\Gamma_\infty^\lambda$ immediately yields corresponding results for $\mathrm{Var}_{\mathrm{sp},\infty}(\lambda)$.
\subsubsection{Preliminaries: matrix order, $M$-matrices, and vectorisation}
\label{subsec: preliminaries matrices}

We collect a few standard facts that will be used to control the sign and monotonicity properties of the stationary covariance. For $A\in\mathbb{R}^{d\times d}$ we denote by $\rho(A)$ its spectral radius,
$$
\rho(A):=\max\left \lbrace \abs{\mu}:\mu\ \text{eigenvalue of }A\right \rbrace,
$$
and we write $A\geq B$ for the component-wise order. Recall that a matrix $A$ is a \emph{$Z$-matrix} if $a_{ij}\leq 0$ for all $i\neq j$. A \emph{(non-singular) $M$-matrix} is a $Z$-matrix of the form $A=sI-B$ with $B\geq0$ and $s>\rho(B)$; equivalently, $A$ is a $Z$-matrix with $\Re(\mu)>0$ for every eigenvalue $\mu$, see \cite{Plemmons1981,Plemmons1994}. The key consequence we will use is the element-wise non-negativity of the inverse.
\begin{theorem}
\label{thm: characterisation invertible M matrix}
Let $A$ be a $Z$-matrix. The following are equivalent:
\begin{enumerate}
\item[(i)] $A$ is a non-singular $M$-matrix,
\item[(ii)] $\Re(\mu)>0$ for every eigenvalue $\mu$ of $A$,
\item[(iii)] $A$ is invertible and $A^{-1}\geq 0$ component-wise.
\end{enumerate}
\end{theorem}
Second, we recall that an \emph{irreducible matrix} $A$ is a square matrix that cannot be transformed, via a permutation matrix $P$ (specifically $PAP^{T}$), into a block upper triangular matrix with more than one diagonal block of positive size. Equivalently, $A$ is irreducible if its associated directed graph is strongly connected. We refer to \cite[Chapter 2.2.] {Plemmons1994} for more details and properties on irreducible matrices

Finally, we fix notation for vectorisation and Kronecker products. Given $A=(v_1\,|\,\dots\,|\,v_d)\in\mathbb{R}^{d\times d}$, we define
$$
\mathrm{vec}(A):=\begin{pmatrix}v_1\\ \vdots\\ v_d\end{pmatrix}\in\mathbb{R}^{d^2}.
$$
For matrices $A\in\mathbb{R}^{m\times n}$ and $B\in\mathbb{R}^{p\times q}$, we denote by $A\otimes B\in\mathbb{R}^{mp\times nq}$ their \emph{Kronecker product}, defined blockwise by
$$
A\otimes B :=
\begin{pmatrix}
a_{11}B & \cdots & a_{1n}B\\
\vdots  & \ddots & \vdots\\
a_{m1}B & \cdots & a_{mn}B
\end{pmatrix}.
$$
We will repeatedly use the identities
\begin{equation}
\label{eq: kron identities}
(A\otimes B)^T=A^T\otimes B^T,
\qquad
\mathrm{vec}(AXB)=(B^T\otimes A)\,\mathrm{vec}(X),
\end{equation}
as well as bilinearity in each argument.

\subsubsection{Covariance dynamics and Lyapunov–vectorised formulation}
In this subsection we derive the evolution equation satisfied by the covariance matrix of the semi-discrete anomaly process. By vectorising the resulting Lyapunov-type equation, we obtain a linear ODE in $\mathbb{R}^{d^2}$, whose long-time limit gives an explicit representation of the stationary covariance.

First, observe that SDE \eqref{eq: SDE anomaly fd} can be rewritten as
$$
dY_t^\lambda  = M^\lambda Y_t^\lambda \, dt  + \Sigma^\lambda(Y_t ^\lambda)  dW_t, \quad Y ^\lambda|_{t = 0} = y_0.
$$
with $M^\lambda:= A_\Delta - B^\lambda $ and $\Sigma^\lambda (y) = \sqrt{ \tau }\,  \mathrm{diag} (D^\lambda y + f^\lambda)L.$ Then the expected value $\mathbb{E}[Y_t^\lambda]$ satisfies
$$
\frac{d}{dt}\mathbb{E}[Y_t ^\lambda] = M^\lambda \mathbb{E}[Y_t ^\lambda], \qquad \mathbb{E}[Y_0^\lambda] = y_0
$$
The solution is
$$
\mathbb{E}[Y_t ^\lambda] = e^{M^\lambda t} y_0.
$$
It is reasonable, as in Section \ref{subsec: stationary variance and monotone dependence on CO2}, that the temperature anomaly goes to zero as $t \to +\infty $. This holds for any initial condition $y_0$ if $M^\lambda $ is negative definite, an assumption that we will always adopt from now on.

We are interested in the covariance matrix
$$
\Gamma_t^\lambda := \mathbb{E}\bigl[(Y_t^\lambda-\mathbb{E}[Y_t^\lambda]) (Y_t^\lambda-\mathbb{E}[Y_t^\lambda])^T\bigr].
$$
For general $y_0$, the covariance dynamics contain an additional (time-dependent) forcing term through $\mathbb{E}[Y_t^\lambda]=e^{M^\lambda t}y_0$.
To keep the covariance equation autonomous and simplify the subsequent vectorised Lyapunov formulation, we
impose the simplifying assumption $y_0=0$, so that $\mathbb{E}[Y_t^\lambda]=0$ for all $t\geq 0$ and thus
$
\Gamma_t^\lambda=\mathbb{E}\bigl[Y_t^\lambda (Y_t^\lambda)^T\bigr].
$
\begin{lemma}
    Assume that $M^\lambda = A_\Delta -B^\lambda$ is negative definite and $y_0 = 0$. Then $ \Gamma_t^\lambda$ satisfies the matrix ODE
    \begin{equation*}
            \begin{split}
                 \frac{d}{dt} \Gamma_t^\lambda &= M^\lambda \Gamma_t^\lambda + \Gamma_t^\lambda (M^\lambda)^T
+ \mathbb{E}\!\left[\Sigma^\lambda(Y_t^\lambda)\,(\Sigma^\lambda(Y_t^\lambda))^T\right]\\
                 & = M ^\lambda \Gamma_t^\lambda + \Gamma_t^\lambda (M^\lambda)^T + \tau \, \sum_{k=1}^d \mathrm{diag}(l^k) (D^\lambda \Gamma_t^\lambda (D^\lambda)^T + f^\lambda (f^\lambda)^T) \mathrm{diag} (l^k),
            \end{split}
        \end{equation*}
        where $l^k$ is the $k$-th column of $L.$
\end{lemma}
\begin{proof}
    Assume for simplicity $\tau = 1$. We drop for the ease of notation the dependence on $\lambda$. Using It\^{o}-formula
    \begin{equation*}
        \begin{split}
            d (Y_t Y_t^T) &= (dY_t)(Y_t)^T + Y_t(dY_t)^T + (dY_t)(dY_t)^T  \\
            & = (MY_t dt + \Sigma  dW_t) (Y_t)^T + Y_t( MY_tdt+ \Sigma  dW_t)^T + (dY_t)(dY_t)^T.
        \end{split}
    \end{equation*}
    Note that, using It\^{o} isometry, we have
    \begin{equation*}
        \begin{split}
            (dY_t)(dY_t)^T =(MY_tdt+ \Sigma  dW_t) (MY_tdt+ \Sigma  dW_t)^T = \Sigma  dW_t (dW_t)^T \Sigma^T.
        \end{split}
    \end{equation*}
    Since $dW_t (dW_t)^T = I dt$, we obtain
    $$
    (dY_t)(dY_t)^T = \Sigma \Sigma^T dt.
    $$
    In conclusion we have
    $$
    d(Y_t Y_t^T) = [M Y_t(Y_t)^T  + Y_t (Y_t)^T M^T + \Sigma \Sigma^T]dt + \widetilde M_t,
    $$
    where $\widetilde M_t$ denotes a martingale. Taking the expected values, we deduce
    $$
    \frac{d}{dt} \Gamma_t = M  \Gamma_t + \Gamma_t  M^T + \mathbb{E}[\Sigma \Sigma^T].
    $$
    But in our setting $\Sigma = \mathrm{diag} (DY_t +f)L$, thus
    $$
    \Sigma \Sigma^T = \mathrm{diag} (DY_t +f) L L^T \mathrm{diag}( DY_t +f)^T.
    $$
    Since
    $$
    L L^T = \sum_{k =1 }^d l^k (l^k)^T,
    $$
    we obtain
    $$
    \Sigma \Sigma^T = \sum_{k = 1}^d \mathrm{diag} ( DY_t + f) l^k (l^k)^T \mathrm{diag} (DY_t +f)^T,
    $$
    Further, using the property
    $$
    \mathrm{diag} (a) b = \mathrm{diag}(b)a, \quad a,b \in \mathbb{R}^d,
    $$
    we have
    $$
    \Sigma \Sigma^T = \sum_{k =1 }^d \mathrm{diag} (l^k ) \left[(DY_t +f) (DY_t +f)^T \right] \mathrm{diag}(l^k).
    $$
    Since $\mathbb{E}[Y_t] = 0$, we have
    $$
    \mathbb{E}\left[(DY_t +f) (DY_t+f)^T \right] = D \Gamma_t  D^T + f f^T,
    $$
    and thus
    $$
    \mathbb{E}\left[\Sigma \Sigma^T \right] = \sum_{k = 1}^d \mathrm{diag}(l^k) (D \Gamma_t  D^T + f f^T) \mathrm{diag} (l^k).
    $$
\end{proof}
In the following technical result, we write the ODE satisfied by the covariance matrix $\Gamma_t $ in more tractable way.
\begin{lemma}
Assume that $M ^\lambda= A_\Delta -B^\lambda$ is negative definite and $y_0 = 0$. The ODE for $\Gamma_t ^\lambda$
$$
 \frac{d}{dt} \Gamma_t^\lambda  = M^\lambda \Gamma_t ^\lambda + \Gamma_t ^\lambda (M^\lambda)^T + \tau \, \sum_{k=1}^d \mathrm{diag}(l^k) (D^\lambda \Gamma_t^\lambda  (D^\lambda)^T + f^\lambda (f^\lambda)^T) \mathrm{diag} (l^k),
$$
is equivalent to
\begin{equation}
\frac{d}{dt} q^\lambda _t = K^\lambda q_t^\lambda  + F^\lambda,
    \label{eq: ODE for covariance matrix}
\end{equation}
with
$$
q_t^\lambda = \mathrm{vec}(\Gamma_t^\lambda ), \quad K^\lambda := I_d \otimes M^\lambda + M^\lambda \otimes I_d + \tau \,  \mathrm{diag} (\mathrm{vec}(C)) (D^\lambda \otimes D^\lambda),
$$
and
$$
F^\lambda = \tau \,  \mathrm{diag} (\mathrm{vec} (C)) \mathrm{vec}(f^\lambda (f^\lambda)^T),
$$
\label{lemma: vectorisation problem in Rd}
\end{lemma}
\begin{proof}
    Assume for simplicity $\tau  = 1$. We drop the dependence on $\lambda$ for the ease of notation. Consider the matrix ODE
$$
\frac{d}{dt} \Gamma_t  = M \Gamma_t  + \Gamma_t   M^T + \sum_{k=1}^d \mathrm{diag}(l^k) \left(D \Gamma_t  D^T + f f^T\right) \mathrm{diag}(l^k).
$$
Vectorising both sides gives
$$
\frac{d}{dt} q_t = \mathrm{vec}(M\Gamma_t ) + \mathrm{vec}(\Gamma_t  M^T)
          + \sum_{k=1}^d \mathrm{vec} \left(\mathrm{diag}(l^k)  \left(D \Gamma_t  D^T + f f^T\right) \mathrm{diag}(l^k)\right).
$$
By the Kronecker product property \eqref{eq: kron identities},
$$
\mathrm{vec}(M\Gamma_t ) = (I\otimes M)q_t,
\qquad
\mathrm{vec}(\Gamma_t  M^T) = (M\otimes I) q_t.
$$
For the noise term, again by \eqref{eq: kron identities},
since $\mathrm{diag}(l^k)^T=\mathrm{diag}(l^k)$,
$$
\mathrm{vec} \left(\mathrm{diag}(l^k) X   \mathrm{diag}(l^k)\right)
= \left(\mathrm{diag}(l^k)\otimes \mathrm{diag}(l^k)\right) \mathrm{vec}(X),
\qquad X\in\mathbb{R}^{d\times d}.
$$
Apply this with $X = D \Gamma_t  D^T + f f^T$ to obtain
$$
\sum_{k=1}^d  \mathrm{vec} \left(\mathrm{diag}(l^k) (D \Gamma_t  D^T + f f^T) \mathrm{diag}(l^k)\right)
= \left(\sum_{k=1}^d \mathrm{diag}(l^k)\otimes \mathrm{diag}(l^k)\right) 
   \mathrm{vec}(D \Gamma_t  D^T + f f^T).
$$
Finally,
$$
\mathrm{vec}(D \Gamma_t  D^T) = (D\otimes D)\, \mathrm{vec}(\Gamma_t ) = (D\otimes D)\,q_t,
\qquad
\mathrm{vec}(f f^T) \ \text{is constant}.
$$
Finally, since $C=LL^T$ and $l^k$ denotes the $k$-th column of $L$, we have
$$
\sum_{k=1}^d  \mathrm{diag}(l^k)\otimes  \mathrm{diag}(l^k)= \mathrm{diag}( \mathrm{vec}(C)),
$$
because the $(i,j)$ entry of $C$ is $C_{ij}=\sum_{k=1}^d l_i^k l_j^k$, and vectorisation turns
entrywise multiplication by $C$ into multiplication by $ \mathrm{diag}( \mathrm{vec}(C))$.
Collecting terms yields
$$
\frac{d}{dt} q_t = \left(I_d\otimes M + M\otimes I_d +  \mathrm{diag}( \mathrm{vec}(C)) (D\otimes D)\right) q_t +\mathrm{diag}( \mathrm{vec}(C))  \mathrm{vec}(f f^T),
$$
which is \eqref{eq: ODE for covariance matrix}.
\end{proof}
Combining the covariance identity from the previous lemma with the vectorisation rule in Lemma \ref{lemma: vectorisation problem in Rd}, we reduce the matrix-valued covariance dynamics to the linear ODE \eqref{eq: ODE for covariance matrix} on $q_t^\lambda = vec (\Gamma_t^\lambda )$; this directly gives the existence and an explicit formula for the stationary covariance under the hypothesis that $K$ is negative-definite, and $C = L L^T$ is non-negative entry-wise.
\begin{corollary}
    Assume that $M^\lambda= A_\Delta -B^\lambda$ is negative definite, $y_0 = 0$ and $$
    K^\lambda= I_d \otimes M^\lambda + M ^\lambda\otimes I_d + \tau \, \mathrm{diag} (\mathrm{vec}(C)) (D^\lambda \otimes D^\lambda)
    $$
    is negative definite.  Assume further that $C = L L^T \geq0 $ component-wise. Then $\lim_{t \to \infty } \Gamma_t^\lambda = :\Gamma_\infty^\lambda$ exists. Further, setting $q^\lambda = \mathrm{vec}(\Gamma_\infty^\lambda)$, we have
    $$
     q ^\lambda= (-K^\lambda)^{-1}F^\lambda,
    $$
    Moreover $(-K^\lambda)^{-1} \geq 0$ component-wise.
    \label{cor: non negative  minus inverse K}
\end{corollary}
\begin{proof}
Assume for simplicity $\tau = 1$. We drop the dependence on $\lambda$ for the ease of notation. The first part of the statement consists in observing that the solution of the ODE \eqref{eq: ODE for covariance matrix} converges if $K$ is negative definite.

Let's now check that, if $K$ is negative definite, then $(-K)^{-1} \geq 0$ component-wise. Write
$$
- K = I_d \otimes (-M) + (-M) \otimes I_d - \mathrm{diag} ( \mathrm{vec}(C)) (D \otimes D).
$$
We claim that $-K$ is a $Z$-matrix, i.e. it has non-positive off-diagonal elements.

First, observe that $-M$ is a $Z$-matrix. Indeed, the discrete Laplacian $A_\Delta$ has non-negative off-diagonal entries and $B $ is diagonal, hence $M = A_\Delta - B$ has non-negative off-diagonals. Therefore $-M$ has non-positive off-diagonals, i.e. $-M$ is a Z-matrix. 

Second, both Kronecker products $I_d \otimes  (-M)$ and $(-M) \otimes  I_d$ have the same sign pattern off-diagonal as $-M$, therefore each of them is a $Z$-matrix. The same holds for their sum

Third, the term $\mathrm{diag} (\mathrm{vec}(C)) ( D \otimes D)$ is diagonal and non-negative. Indeed, since $D$ is diagonal, then $D \otimes D$ is diagonal. Hence $\mathrm{diag} (\mathrm{vec} (C)) ( D \otimes D)$ is diagonal as well. Moreover, $C \geq 0$ entry-wise implies $\mathrm{diag} (\mathrm{vec} (C)) \geq 0$ entry-wise; therefore  $ \mathrm{diag} (\mathrm{vec}(C)) (D \otimes D)$ has non-negative diagonal entries. Consequently,
$\mathrm{diag}(\mathrm{vec}(C))(D\otimes D)$ only affects the diagonal of $-K$ and does not change any
off-diagonal sign.

In conclusion, we have proved that $-K$ is a $Z$-matrix. By the hypothesis on $K$, we have that $-K$ is positive definite. By Theorem \ref{thm: characterisation invertible M matrix}, we deduce that $(-K)^{-1} \geq 0$ component-wise.

\end{proof}

\subsection{Monotonicity of the stationary covariance in the radiative forcing parameter}
\label{subsec: monotone covariance}

We now study how the stationary covariance of the semi-discrete temperature anomaly depends on the forcing
parameter $\lambda$. For each $\lambda$, let $Y_t^\lambda\in\mathbb{R}^d$ solve the finite-difference SDE
\eqref{eq: SDE anomaly fd} with $y_0=0$, i.e.
\begin{equation*}
dY_t^\lambda = M^\lambda Y_t^\lambda\,dt
+ \sqrt{\tau}\, \mathrm{diag} \bigl(D^\lambda Y_t^\lambda + f^\lambda\bigr)\,L\,dW_t,
\qquad
Y_0^\lambda=0,
\end{equation*}
where
$$
M^\lambda := A_\Delta - B^\lambda,
\qquad
B^\lambda := \mathrm{diag}(b^\lambda),\quad D^\lambda := \mathrm{diag}(d^\lambda),\quad f^\lambda\in\mathbb{R}^d,
\quad C:=LL^T .
$$
Recall that the coefficients are obtained by sampling the equilibrium profile $T^{*,\lambda}$ on the interior
grid points $z_m$:
\begin{equation}
\label{eq:coeffs-sampled}
b_m^\lambda = a^\lambda(z_m)
= - Q(z_m) \beta' \left(T^{*,\lambda}(z_m)\right) + r_1,
\quad
d_m^\lambda = \beta' \left(T^{*,\lambda}(z_m)\right),
\qquad
f_m^\lambda = \beta \left(T^{*,\lambda}(z_m)\right).
\end{equation}
Here we used the affine form $\mathcal R_e(T)=r_0+r_1T$ so that $\mathcal R_e'(T)\equiv r_1$. Recall that, under the non-restrictive assumption $y_0 =0$,
$$
\Gamma_t ^\lambda = \mathbb{E}\left[Y_t^\lambda (Y_t^\lambda)^T\right],
\qquad
\Gamma _\infty^\lambda = \lim_{t\to\infty} \Gamma_t ^\lambda,
$$
whenever the limit exists. By Lemma~\ref{lemma: vectorisation problem in Rd} and
Corollary~\ref{cor: non negative  minus inverse K}, under the stability assumptions on $M^\lambda$ and $K^\lambda$
we have
\begin{equation*}
q_\infty^\lambda := \mathrm{vec}(\Gamma _\infty^\lambda) = (-K^\lambda)^{-1}F^\lambda,
\end{equation*}
where
\begin{equation}
\label{eq:KFlambda}
\quad
F^\lambda = \tau\,\mathrm{diag}(\mathrm{vec}(C)) \, \mathrm{vec}(f^\lambda (f^\lambda)^T).
\end{equation}

As in the zero-dimensional case, one expects $\lambda\mapsto T^{*,\lambda}(x)$ to be increasing for each $x\in \mathcal{D}$
(see Section~\ref{sec: one}). Since, for $T \in (T_l, T_u),$ $\beta$ is increasing, this implies that $\lambda\mapsto f^\lambda$ is
increasing component-wise.

To control the $\lambda$-dependence of $B^\lambda = \mathrm{diag}(b^\lambda)$ and $D^\lambda =\mathrm{diag} ((\beta'(T^{*,\lambda}(x_i))_i)) $, we assume, as above, that the equilibrium temperatures
remain in a regime where $\beta$ is increasing, e.g.
\begin{equation}
\label{eq:Tref-assumption}
T^{*,\lambda}(z_m) \in (T_l, T_u),\qquad m=1,\dots,d,
\end{equation}
so that $\partial_ \lambda b^\lambda=0$ and
$\beta'(T^{*,\lambda})\equiv \tfrac{\beta_{\max} - \beta_{\min}}{T_u - T_l}$.
\begin{proposition}[Strict component-wise increase of the stationary covariance]
\label{prop:monotone-covariance}
Let
$$
s:=\frac{\beta_{\max}-\beta_{\min}}{T_u-T_l}>0 .
$$
Assume that, for every $\lambda$ under consideration, the equilibrium profile satisfies
$$
T^{*,\lambda}(x)\in(T_l,T_u)\qquad \text{for all }x\in\mathcal D,
$$
so that $\beta'(T^{*,\lambda}(x))\equiv s$ on $\mathcal D$.
Assume moreover the coercivity condition
\begin{equation}
\label{eq:coercivity}
r_1-Q(x) s \geq 0\qquad\text{in }\mathcal D.
\end{equation}
Let $C=LL^T$ and assume
$$
C\geq 0\ \text{entry-wise},\qquad C\neq 0.
$$
Assume that $K^\lambda$ is negative definite, where
$$
M^\lambda:=A_\Delta-B^\lambda,
\qquad
K^\lambda:= I_d\otimes M^\lambda + M^\lambda\otimes I_d
+\tau \,\mathrm{diag}(\mathrm{vec}(C))\,(D^\lambda\otimes D^\lambda).
$$
Then
$
\partial_\lambda \Gamma_\infty^\lambda > 0$ entry-wise.
\end{proposition}

\begin{proof}
\emph{Step 1: $\partial_\lambda T^{*,\lambda}>0$ in the interior.} Set $u^\lambda:=\partial_\lambda T^{*,\lambda}$. Differentiating the elliptic equilibrium problem
\eqref{eq: elliptic equilibrium dirichlet} with respect to $\lambda$ yields
$$
\Delta u^\lambda(x) + Q(x)\beta'(T^{*,\lambda}(x)) u^\lambda(x) + 1 -r_1 u^\lambda(x)=0,
\qquad
u^\lambda|_{\partial\mathcal D}=0.
$$
Since $T^{*,\lambda}(x)\in(T_l,T_u)$ for all $x$ and $\beta$ is affine on $(T_l,T_u)$, we have
$\beta'(T^{*,\lambda}(x))\equiv s$. Hence
$$
-\Delta u^\lambda(x) + (r_1-Q(x)s) u^\lambda(x) = 1,
\qquad
u^\lambda|_{\partial\mathcal{D}}=0.
$$
By \eqref{eq:coercivity} and the maximum principle, see \cite[Section 6.4]{Evans2022}, it follows that
$$
u^\lambda(x)>0\quad \text{in } \mathcal D,
$$
and in particular $u^\lambda(z_m)>0$ for all interior grid points $z_m$.

\emph{Step 2: $\partial_\lambda f^\lambda>0$ component-wise.}
Recall $f_m^\lambda=\beta(T^{*,\lambda}(z_m))$. Since $\beta'(T^{*,\lambda}(z_m))\equiv s$,
$$
\partial_\lambda f_m^\lambda
=\beta'(T^{*,\lambda}(z_m)) \partial_\lambda T^{*,\lambda}(z_m)
= s u^\lambda(z_m) > 0,
\qquad m=1,\dots,d.
$$
Also $f^\lambda>0$ component-wise since $\beta\in(0,1)$.

\emph{Step 3: $K^\lambda$ is $\lambda$-independent on $(T_l,T_u)$.}
Because $\beta'(T^{*,\lambda}(z_m))\equiv s$, we have $D^\lambda=sI$ and
$B^\lambda=\mathrm{diag}(-Q(z_m)s+r_1)$, hence $D^\lambda, M^\lambda$ and $K^\lambda$ do not depend on $\lambda$
in this regime. We denote them simply by $D, M$ and $K$.

\emph{Step 4: $-K$ is an irreducible nonsingular $M$-matrix and $(-K)^{-1}>0$ entry-wise.}
Since $A_\Delta$ is the standard finite-difference Laplacian (five-point stencil) on a connected interior grid,
its directed graph is strongly connected, hence $A_\Delta$ is irreducible. As $B^\lambda$ is diagonal,
$M^\lambda=A_\Delta-B^\lambda$ has the same off-diagonal pattern as $A_\Delta$, so $-M$ is an irreducible $Z$-matrix. Consider the Kronecker sum
$$
S:= I_d\otimes(-M)+(-M)\otimes I_d.
$$
Then $S$ is a $Z$-matrix and is irreducible\footnote{Indeed, write the index set of $\R^{d^2}$ as pairs
$(i,j)$ with $i,j\in\{1,\dots,d\}$. By definition of the Kronecker sum,
$S^\lambda$ allows moves that change \emph{only one coordinate at a time}:
from $(i,j)$ to $(i',j)$ whenever $(-M)_{ii'}\neq 0$, and from $(i,j)$ to $(i,j')$ whenever $(-M)_{jj'}\neq 0$.
Since $-M$ is irreducible, its graph is strongly connected, hence any pair $(i,j)$ can reach any other pair
$(i',j')$ by first moving in the $i$-coordinate and then in the $j$-coordinate. Therefore $S$ is irreducible.}. Moreover, since $D$ is diagonal, $D\otimes D$ is diagonal, hence
$\mathrm{diag}(\mathrm{vec}(C))(D\otimes D)$ is diagonal. Therefore
$$
-K = S - \tau\,\mathrm{diag}(\mathrm{vec}(C))(D\otimes D)
$$
is still a $Z$-matrix and remains irreducible (diagonal modifications do not affect irreducibility). By assumption, $K$ is negative definite, hence $-K$ is positive definite and in particular all eigenvalues
of $-K$ have strictly positive real part. Since $-K$ is a $Z$-matrix, this implies that $-K$ is a nonsingular
$M$-matrix. Together with irreducibility, it follows, thanks to \cite[Corollary 3.8, Chapter 5]{Plemmons1981}, that
$$
(-K)^{-1}>0 \qquad\text{entry-wise.}
$$

\emph{Step 5: differentiate the stationary covariance formula and conclude strict positivity.}
By Lemma \ref{lemma: vectorisation problem in Rd} and Corollary \ref{cor: non negative  minus inverse K},
$$
q_\infty^\lambda:=\mathrm{vec}(\Gamma_\infty^\lambda)=(-K)^{-1}F^\lambda,
\qquad
F^\lambda=\tau\,\mathrm{diag}(\mathrm{vec}(C)) \mathrm{vec}(f^\lambda(f^\lambda)^T).
$$
Since $K$ is constant in $\lambda$, differentiation gives
$$
\partial_\lambda q_\infty^\lambda = (-K)^{-1}\,\partial_\lambda F^\lambda,
$$
and
$$
\partial_\lambda F^\lambda
=\tau\,\mathrm{diag}(\mathrm{vec}(C))\,\mathrm{vec}((\partial_\lambda f^\lambda)(f^\lambda)^T
+ f^\lambda(\partial_\lambda f^\lambda)^T).
$$
By Step 2, both $f^\lambda$ and $\partial_\lambda f^\lambda$ are strictly positive component-wise, hence the matrix
$$
G^\lambda := (\partial_\lambda f^\lambda)(f^\lambda)^T + f^\lambda(\partial_\lambda f^\lambda)^T
$$
is strictly positive entry-wise, so $\mathrm{vec}(G^\lambda)>0$ component-wise.
Since $C\geq 0$ entry-wise, $\mathrm{diag}(\mathrm{vec}(C))\geq 0$ entry-wise, and because $C\neq 0$ we have
$\mathrm{diag}(\mathrm{vec}(C))\,\mathrm{vec}(G^\lambda)\neq 0$. Therefore
$$
\partial_\lambda F^\lambda \geq 0,\qquad \partial_\lambda F^\lambda\neq 0.
$$
Finally, since $(-K)^{-1}>0$ entry-wise (Step 4), it follows that
$$
\partial_\lambda q_\infty^\lambda = (-K)^{-1}\,\partial_\lambda F^\lambda > 0
\quad\text{component-wise},
$$
and hence $\partial_\lambda \Gamma_\infty^\lambda>0$ entry-wise.
\end{proof}
In particular, since the stationary spatial variance is the trace of $\Gamma ^\lambda_\infty $ (see Definition \ref{def: spatial variance}), the previous result guarantees that it is monotone increasing with respect to $\lambda$.
\begin{corollary}[Monotonicity of the stationary spatial variance]
Under the assumptions of Proposition \ref{prop:monotone-covariance}, the stationary spatial variance
\[
\mathrm{Var}_{\mathrm{sp},\infty}(\lambda)=\Tr(\Gamma_\infty^\lambda)
\]
is monotone increasing in $\lambda$.
\end{corollary}
\begin{remark}[Stronger forcing, stronger climate spatial co-variability]
We remark that Proposition \ref{prop:monotone-covariance} states that
$$
\partial_\lambda (\Gamma_\infty^\lambda)_{ij}>0\qquad \text{for all } i,j,
$$
and therefore, for any two distinct grid points $z_i\neq z_j$,
$$
\partial_\lambda  \mathrm{Cov} \left(Y_\infty^\lambda(z_i),Y_\infty^\lambda(z_j)\right)>0.
$$
In climate terms, increasing radiative forcing increase not only the amplitude of local fluctuations (variance or spatial variance) but also their \emph{spatial co-fluctuation}: anomalies at different locations tend to move more ``together'' in the sense of increasing covariance.

This covariance-amplification effect is complementary to recent statistical work on sphere-cross-time temperature fields which studies temporal evolution (and possible non-stationarity) in the spherical-harmonic coefficients and develops rigorous change-point methodology; see, e.g., \cite{Caponera2021,Caponera2023, Marinucci2021,Marinucci2024}. These contributions focus on detecting departures from stationarity in observed data, while our result provides a forcing-driven mechanism producing monotone covariance amplification in a reduced stochastic EBM.
\end{remark}
Note that among the assumptions of Proposition~\ref{prop:monotone-covariance} there is $C = L L^T \geq0$ element-wise. This is mandatory. Indeed, in the next example, we show that, if the spatial covariance matrix $C$ of the noise has even one negative entry, then the monotonicity of the spatial variance
may fail, even when the noise amplitude is increasing componentwise in $\lambda$.

 \begin{example}[Failure of monotonicity with negative correlations]
Consider the two-dimensional SDE with additive noise
$$
dY_t^\lambda = M Y_t^\lambda\,dt + \Sigma^\lambda\, dW_t,
\qquad Y_0^\lambda=0,
$$
where $W_t$ is a $2$-dimensional Brownian motion and
$$
\Sigma^\lambda := \mathrm{diag}(f^\lambda)\,L,
\quad  
C:=LL^T.
$$
We choose
$$
M=\begin{pmatrix}-1&s\\ s&-1\end{pmatrix},
\qquad 0<s<1,
$$
so that $M$ is symmetric negative definite (its eigenvalues are $-1\pm s<0$).
Let
$$
C=\begin{pmatrix}1&-c\\ -c&1\end{pmatrix},
\qquad 0<c<1,
$$
which is positive definite (its eigenvalues are $1\pm c > 0$) but has
a \emph{negative} off-diagonal entry $C_{12}=-c<0$. Let $L \in \mathbb{R}^{2 \times 2}$ such that $L L^T = C$; note that $L$ exists, for instance, thanks to the Cholesky factorisation. Finally, define 
$$
f^\lambda=\begin{pmatrix}\lambda\\ 1\end{pmatrix},
\qquad \lambda\geq 0,
$$
which is increasing componentwise in $\lambda$.

Let $\Gamma_t ^\lambda:=\mathbb{E}[Y_t^\lambda (Y_t^\lambda)^T]$. Since the noise is additive, It\^{o} formula gives that $\Gamma_t $ solves the ODE
$$
\frac{d}{dt} \Gamma_t ^\lambda = M \Gamma_t ^\lambda + \Gamma_t ^\lambda M^T + S^\lambda,
\quad 
S^\lambda:=\Sigma^\lambda(\Sigma^\lambda)^T.
$$
Because $\Sigma^\lambda= \mathrm{diag}(f^\lambda)L$ and $LL^T=C$, we have
$$
S^\lambda
=  \mathrm{diag}(f^\lambda)\,C\, \mathrm{diag}(f^\lambda)
=
\begin{pmatrix}\lambda^2&-c\lambda\\ -c\lambda&1\end{pmatrix}.
$$
Since $M$ is negative definite, $\Gamma_t ^\lambda$ converges to the unique stationary covariance
$\Gamma_\infty^\lambda$ solving
\begin{equation*}
M \Gamma _\infty^\lambda + \Gamma_\infty^\lambda M + S^\lambda = 0.
\end{equation*}
The stationary covariance matrix $\Gamma_\infty^\lambda$ can also be expressed as
$$
\Gamma_\infty ^\lambda = \int_0^\infty e^{M u} S^\lambda  e^{Mu} du.
$$
Let $\sigma_+ := -1+s$ and $\sigma_- = -1 -s$ be the eigenvalues, and let $v_+,v_-$ be eigenvectors corresponding to respectively to $\sigma_+$ and $\sigma_-$. We can take them orthonormal and given by
$$
v_+ = \frac{1}{\sqrt{2}}\begin{pmatrix}
    1 \\ 1
\end{pmatrix}, \qquad v_{-} =\frac{1}{\sqrt{2}} \begin{pmatrix}
    1 \\ -1.
\end{pmatrix}
$$
Let
$$
V:= \begin{pmatrix}
\vert & \vert \\
    v_+ & v_- \\
    \vert & \vert 
\end{pmatrix}
$$
in such a way that $V V^T = V^T V = I.$ Hence
$$
V^T M V = \begin{pmatrix}
    -1+ s & 0 \\
    0 & -1-s
\end{pmatrix}:= D_M.
$$
Recall that
$$
e^{Mu} = V e^{D_M u} V^T, \quad e^{D_M u } = \begin{pmatrix}
    e^{\sigma_+ u} & 0 \\
    0 & e^{\sigma_- u}
\end{pmatrix}.
$$
Using the linearity and cyclicity of the trace
\begin{equation*}
    \begin{split}
        \mathrm{Tr}  (\Gamma_\infty ^\lambda ) &= \mathrm{Tr}  (\int_0^\infty e^{M u} S^\lambda  e^{Mu} du) = \int_0^\infty \mathrm{Tr} ( e^{M u} S^\lambda e^{Mu} ) du = \int_0^\infty \mathrm{Tr}  ( V e^{D_M u}V^T S^\lambda V e^{D_M u} V^T ) du\\
        & =\int_0^\infty \mathrm{Tr}  ( V^TV e^{D_M u}V^T S^\lambda V e^{D_M u}  ) du \\
        & = \int_0^\infty \mathrm{Tr}  (e^{D_M u} V^T S^\lambda V e^{D_M u}) du \\
        &= \int_0^\infty  \mathrm{Tr}  (e^{2D_M u} V^T S^\lambda V) du.
    \end{split}
\end{equation*}
Setting $\tilde{ S}^\lambda := V^T S^\lambda V$, then
\begin{equation*}
    \begin{split}
        \mathrm{Tr}  (\Gamma_\infty ^\lambda ) &= \int_0^\infty \mathrm{Tr}  (e^{2 D_M u}\tilde{ S}^\lambda ) du = \int_0^\infty \left( e^{2  u\sigma_+ } \tilde S_{11}^\lambda  + e^{2  u\sigma_- } \tilde S_{22}^\lambda \right) du  = -\frac{\tilde S_{11}^\lambda}{2\sigma_+} - \frac{\tilde S_{22}^\lambda}{2\sigma_-}.
    \end{split}
\end{equation*}
Since
$$
\tilde{S}_{11}^\lambda  =  \frac{\lambda^2 - 2c \lambda +1}{2}, \quad \tilde{S}_{22}^\lambda  = \frac{\lambda^2 + 2 c \lambda +1}{2},
$$
we obtain, after some computations
$$
\mathrm{Tr} (\Gamma_\infty ^\lambda) = \frac{\lambda^2 - 2cs \lambda +1}{2(1-s^2)}.
$$
Computing the derivative with respect to $\lambda$, we observe that
$$
\partial_\lambda \mathrm{Tr}(\Gamma_\infty ^\lambda) = \frac{ \lambda - cs}{1-s^2}.
$$
Since $s \in (0,1)$, for $\lambda < cs $ the spatial variance, i.e. the trace of the stationary covariance matrix, is decreasing.

\end{example}

\section{Analysis of observational sea surface temperature data}
\label{sec: three}
\begin{figure}
    \centering
    \includegraphics[width=1\linewidth]{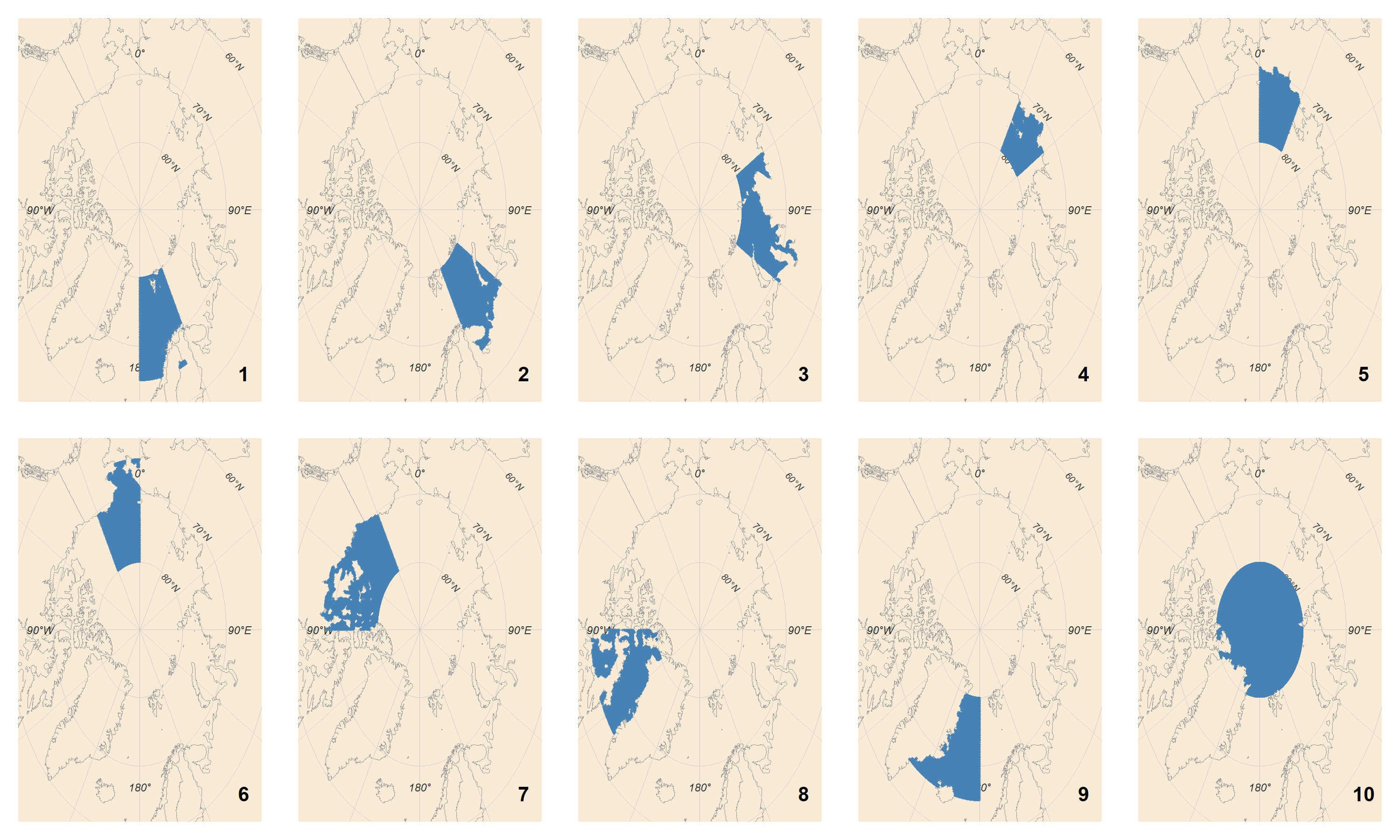}
    \caption{Geographical distribution of the Arctic Ocean and its marginal seas. Numbered labels indicate: (1) Norwegian Sea, (2) Barents Sea, (3) Kara Sea, (4) Laptev Sea, (5) East Siberian Sea, (6) Chukchi Sea, (7) Beaufort Sea, (8) Baffin Bay, (9) Greenland Sea, and (10) Arctic Ocean.}
    \label{fig: area_arctic}
\end{figure}

In this section we compare our results with data. The real-world data are the product of the net outcome of all the geophysical processes at work within the system and from this perspective, the model presented is highly simplified. We focus on Surface Sea Temperature (SST) during summer (August) in the northern high latitudes, above 60°N, where the ice-albedo feedback is predominant. 

In glacial zones small temperature increments can trigger powerful feedbacks that further accelerate the initial warming. In particular, the ice-albedo feedback represents a key mechanism governing Arctic SST, especially during summer months \cite{NOAA_SST_23}.  
When the temperature is near a certain threshold, a small temperature increase triggers ice melting, giving way to liquid water; since water has a much lower albedo, this leads to increased heat absorption and then to higher temperatures. \\
Due to the increase in $\mathrm{CO}_2$, Arctic regions are experiencing the so-called \textit{Arctic Amplification} 
\cite{Serreze} and a severe retreat of sea ice 
\cite{Goosse2009}. Therefore, August SSTs are currently exhibiting higher values than in the past 
\cite{CarvalhoWang2020} 
and northern seas are undergoing a condition of abrupt interannual changes \cite{LandrumHolland2020}
, leading to greater fluctuations in SST \cite{NOAA_SST_23,Alexander2018}. This is the dynamic described by our model in Section \ref{sec: one}.

Concerning Section \ref{sec: two}, we addressed the concept of spatial variance. By this term, we intend to examine how the local mechanism described above manifests on a larger scale. This reflects the spatial uniformity of ice-water dynamics across the entire ocean, clarifying whether the process is homogeneous or characterised by a patchwork distribution \cite{Barber2009,StrongRigor2013,Timmermans2015}. In other words, we are assessing the heterogeneity of the phenomenon, looking at the extent to which the system is fragmented into patches, through the spatial variability of SST. In particular, we have shown that temperature variability, under a certain regime, also increases spatially, implying a patchy behaviour and non-homogeneous distribution.

We used observations from NOAA (National Oceanic and Atmospherical Administration), in particular NOAA OI SST V2 High Resolution Dataset (\url{https://psl.noaa.gov/data/gridded/data.noaa.oisst.v2.highres.html}), provided by the NOAA PSL, Boulder, Colorado, USA (\url{https://psl.noaa.gov}) \cite{Huang2021}. The dataset employed to assess temporal variance analysis contains daily SST values from 1981/09 to 2026/01, with Spatial Coverage of 0.25 degree latitude x 0.25 degree longitude global grid (1440x720). For the spatial variance analysis, we utilised monthly mean SSTs, sharing the same timeframe and geographical extent as the aforementioned dataset.

Initially, to detect signals of temporal SST variability, we first fixed a study area. Than, we extracted data for every August between 1982 and 2025 and we computed the daily spatial mean for each selected month. Subsequently, we calculated statistics (histograms, mean, standard deviation) over overlapping 10-year periods ($31\times10$ data for each period), moving the window forward by one year at a time. 

\begin{figure}
    \centering
    \includegraphics[width=1.1\linewidth]{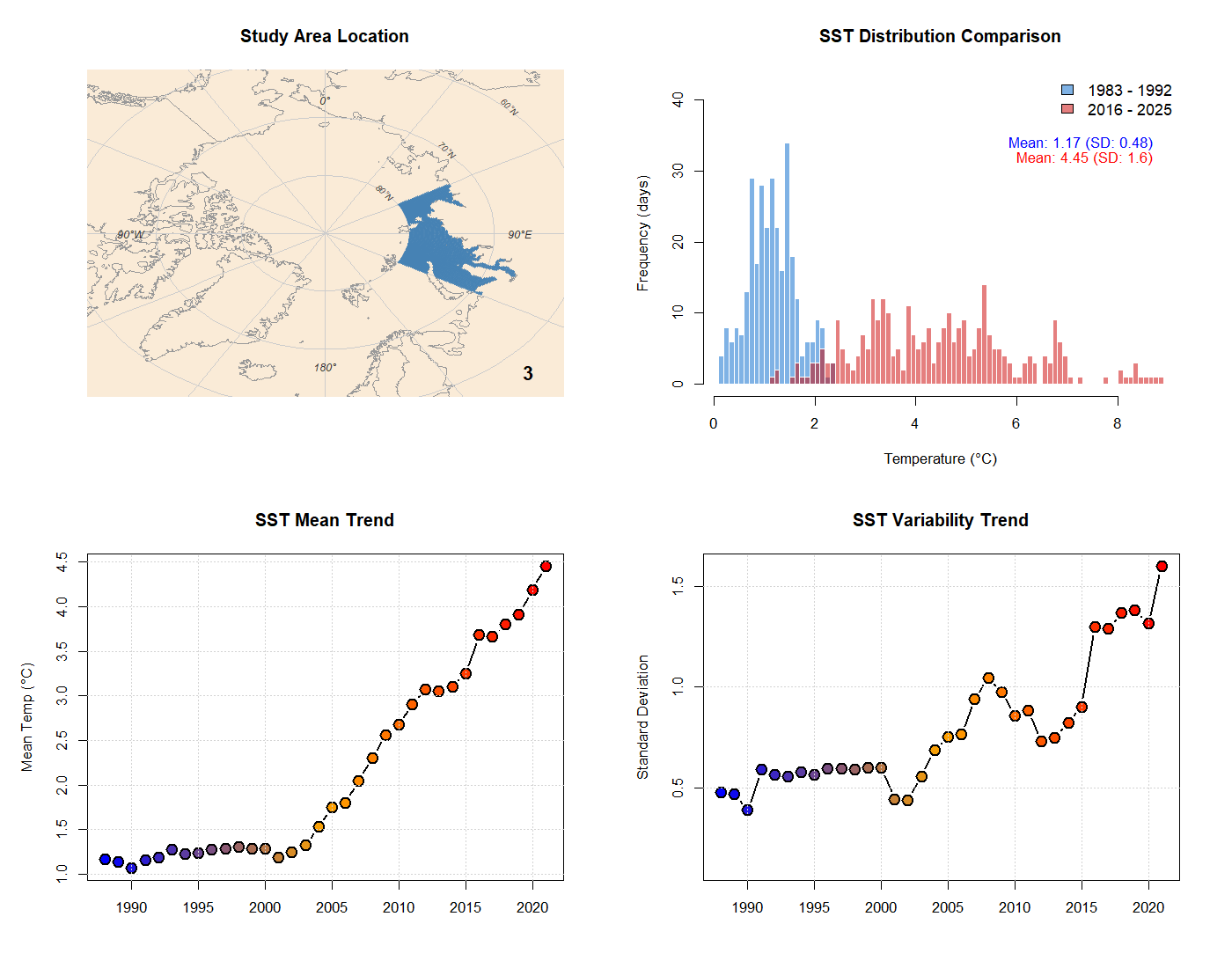}
    \caption{Kara Sea. Top-left: location of the study area. Top-right: comparison between the histograms (number of days) of the spatial-averaged daily August SST for the first decade (1983-1992, blue) and the last calculable decade (2016-2025, red), along with their respective mean and standard deviation values. Bottom-left: decadal spatial mean trend of August SST (centred on the mid-year of each decade, one year moving window). Bottom-right: decadal spatial standard deviation trend of August SST (centred on the mid-year of each decade, one year moving window). The blue-to-red colour scale represents the progression of time across the decades, in accordance with the colour scheme used in the histogram.}
    \label{fig: var_dec_kara}
\end{figure}

\begin{figure}
    \centering
    \includegraphics[width=1.1\linewidth]{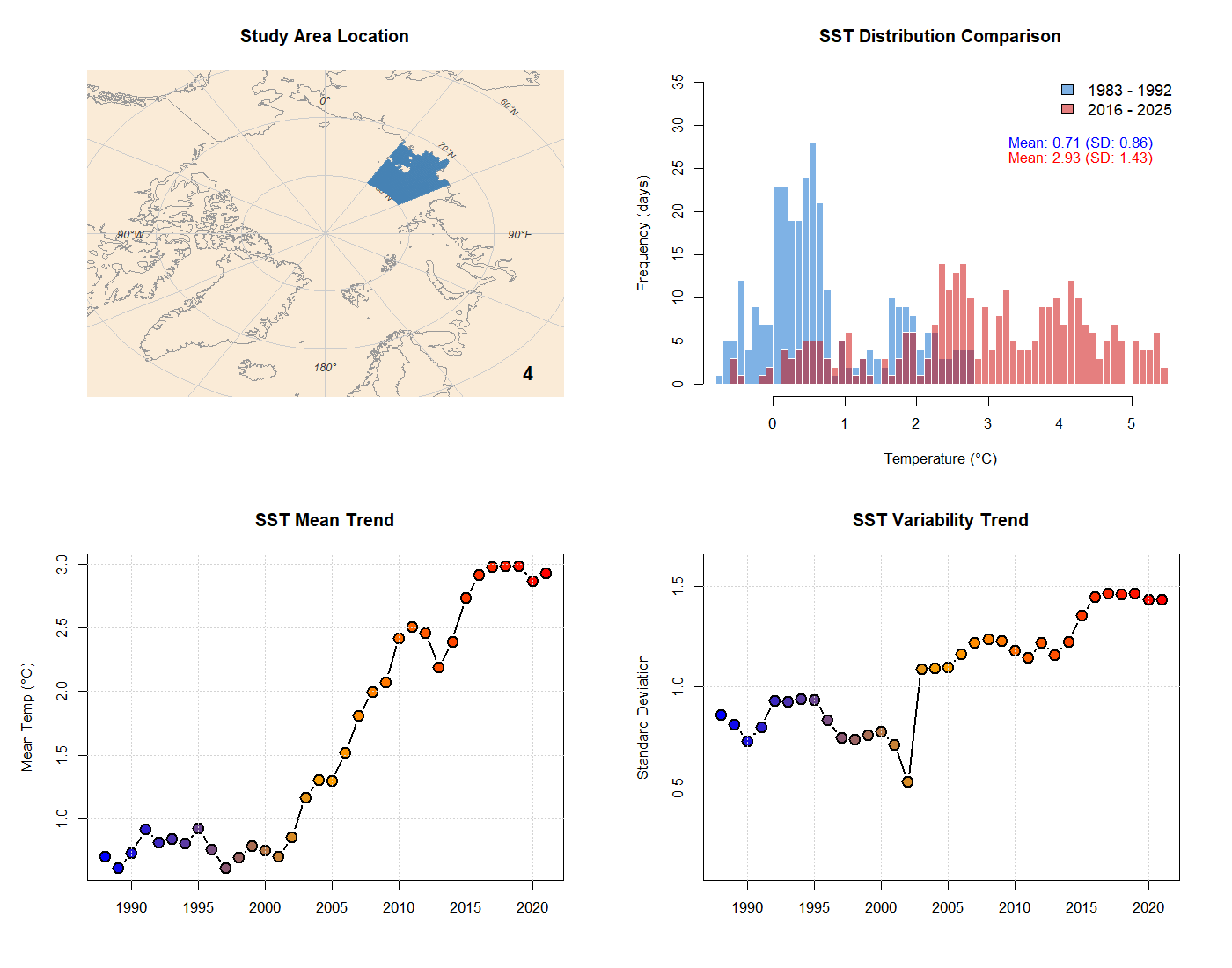}
    \caption{Laptev Sea. Top-left: location of the study area. Top-right: comparison between the histograms (number of days) of the spatial-averaged daily August SST for the first decade (1983-1992, blue) and the last calculable decade (2016-2025, red), along with their respective mean and standard deviation values. Bottom-left: decadal spatial mean trend of August SST (centred on the mid-year of each decade, one year moving window). Bottom-right: decadal spatial standard deviation trend of August SST (centred on the mid-year of each decade, one year moving window). The blue-to-red colour scale represents the progression of time across the decades, in accordance with the colour scheme used in the histogram.}
    \label{fig: var_dec_laptev}
\end{figure}

\begin{figure}
    \centering
    \includegraphics[width=1.1\linewidth]{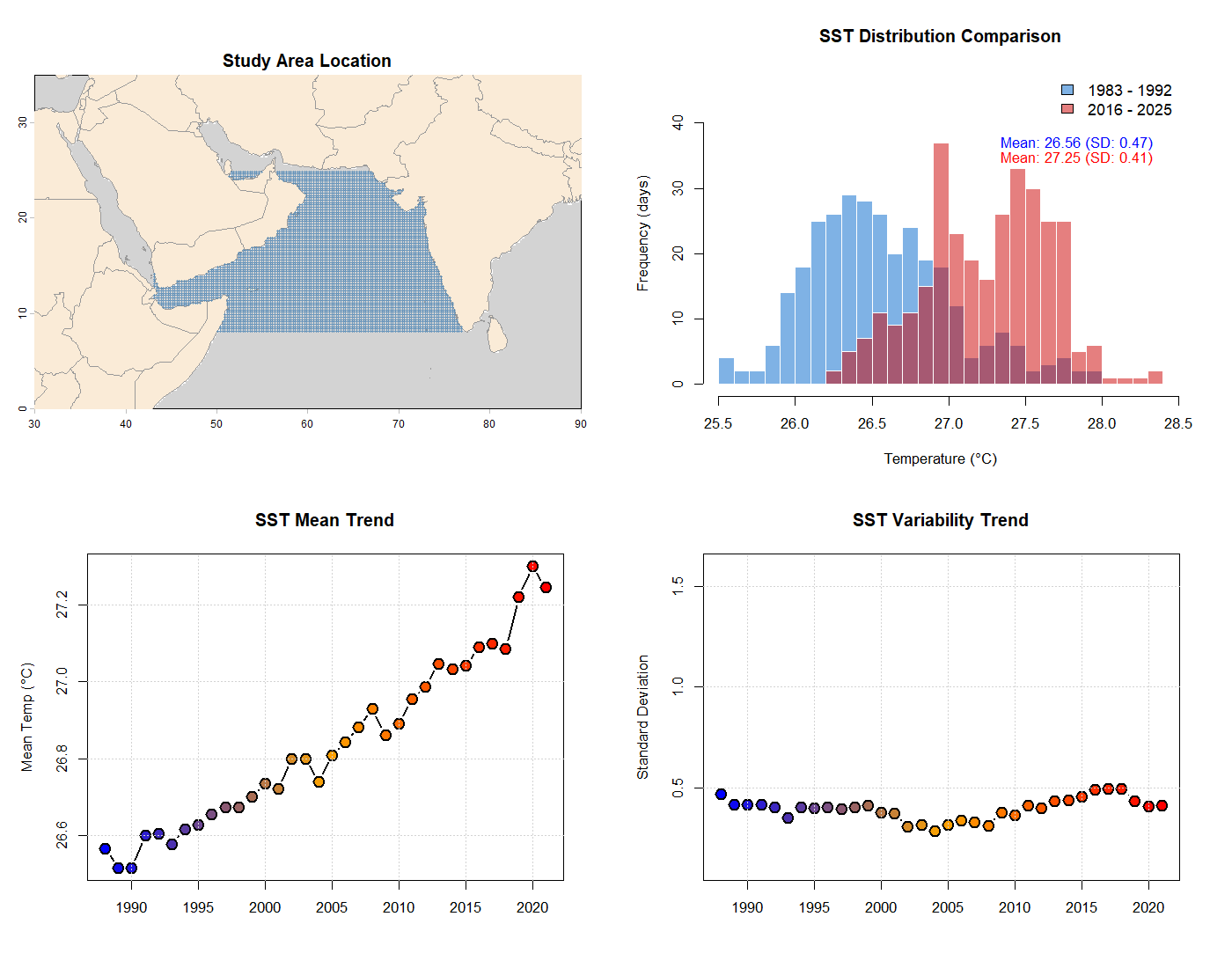}
    \caption{Arabic Sea. Top-left: location of the study area (long-lat). Top-right: comparison between the histograms (number of days) of the spatial-averaged daily August SST for the first decade (1983-1992, blue) and the last calculable decade (2016-2025, red), along with their respective mean and standard deviation values. Bottom-left: decadal spatial mean trend of August SST (centred on the mid-year of each decade, one year moving window). Bottom-right: decadal spatial standard deviation trend of August SST (centred on the mid-year of each decade, one year moving window). The blue-to-red colour scale represents the progression of time across the decades, in accordance with the colour scheme used in the histogram. }
    \label{fig: var_dec_arabic_sea}
\end{figure}

In Figure \ref{fig: var_dec_kara} we show the case of the Kara Sea, a marginal sea of the Arctic Ocean. This example is consistent with our results. The top-left panel indicates the location of the study area. The top-right panel shows a comparison between the histograms of the spatial-averaged daily August SST for the first (blue) and the last (red) calculable decade, along with their respective mean and variance. We note that the distribution has evolved significantly over time: the mean is increased and the tails of the distribution have broadened, as confirmed by the increase in the mean value and standard deviation value, respectively. The bottom-left panel shows the decadal spatial mean of August SST over time (centred on the mid-year of each decade) and a significant increasing trend can be observed. The bottom-right panel shows the decadal spatial standard deviation of August SST over time (centred on the mid-year of each decade). After a stationary phase, also observable in the mean trend, a period of growth can be distinctly discerned, i.e. an increased variability. Similar observations can be done for Figure \ref{fig: var_dec_laptev}, which illustrates the Laptev Sea case \cite{Shabanov2024}.

Comparing the Arabic Sea (Figure \ref{fig: var_dec_arabic_sea}) to the previous cases, the growth trend of the mean is considerably weaker and the distribution tails do not change in shape. Consequently, the standard deviation shows no growth and remains steady. This could possibly be explained by the fact that the equilibrium temperature of that area far exceeds the reference interval, meaning we are referring to the right-side plateau of the co-albedo curve 
(equation \eqref{eq: piecewise co-albedo}), i.e. in \eqref{prop: variance derivative} $\mathrm{Var}_\infty'\equiv 0$.

\begin{figure}
    \centering
    \includegraphics[width=.9\linewidth]{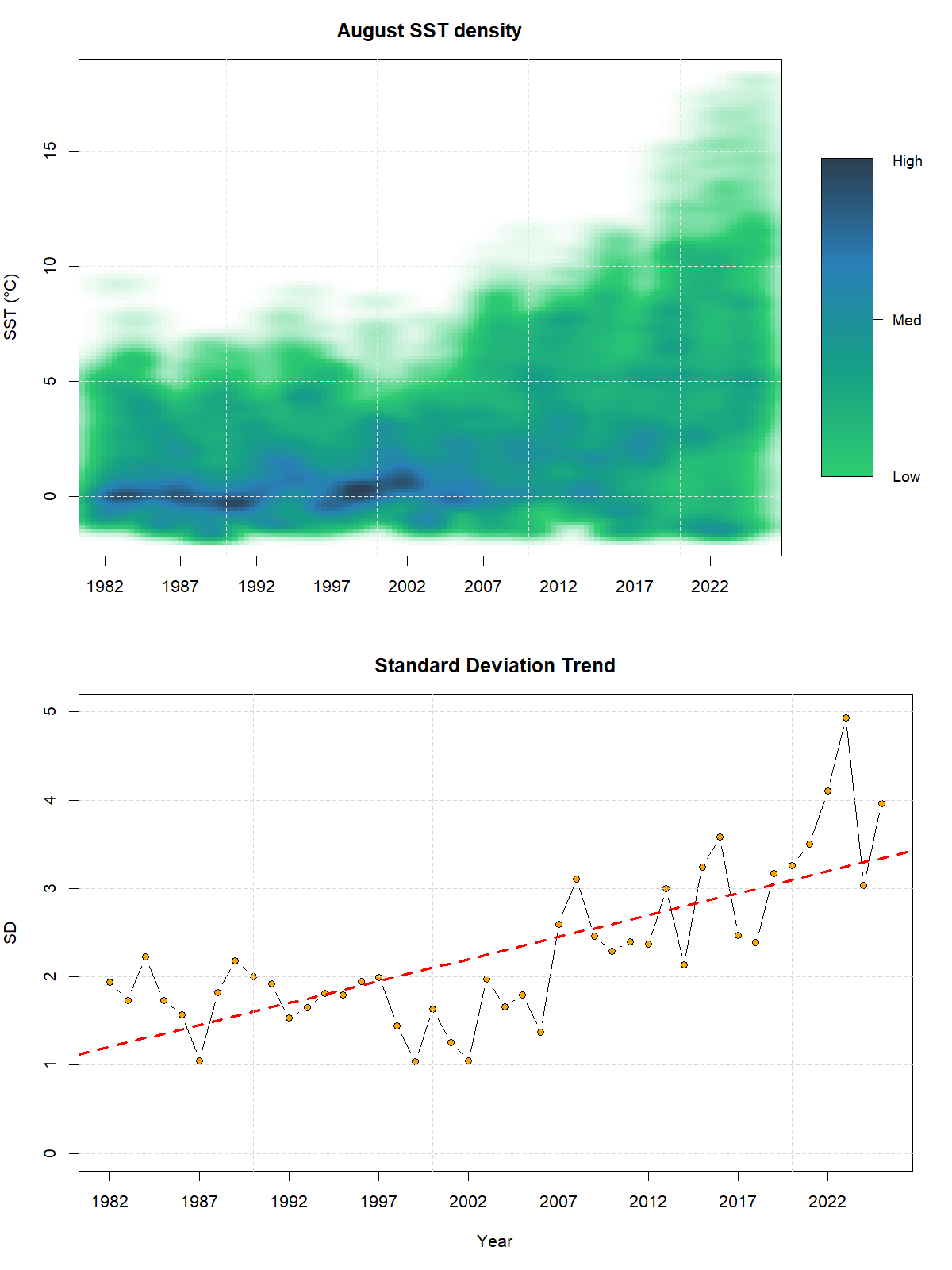}
    \caption{Kara Sea (panel 3, Figure \ref{fig: area_arctic}). Top: density distribution of annual August mean SST values across all individual grid points of the study area. Bottom: interannual standard deviation (dashed line with markers), together with its growth trend (dashed red line).}
    \label{fig: var_spa_kara}
\end{figure}

\begin{figure}
    \centering
    \includegraphics[width=.9\linewidth]{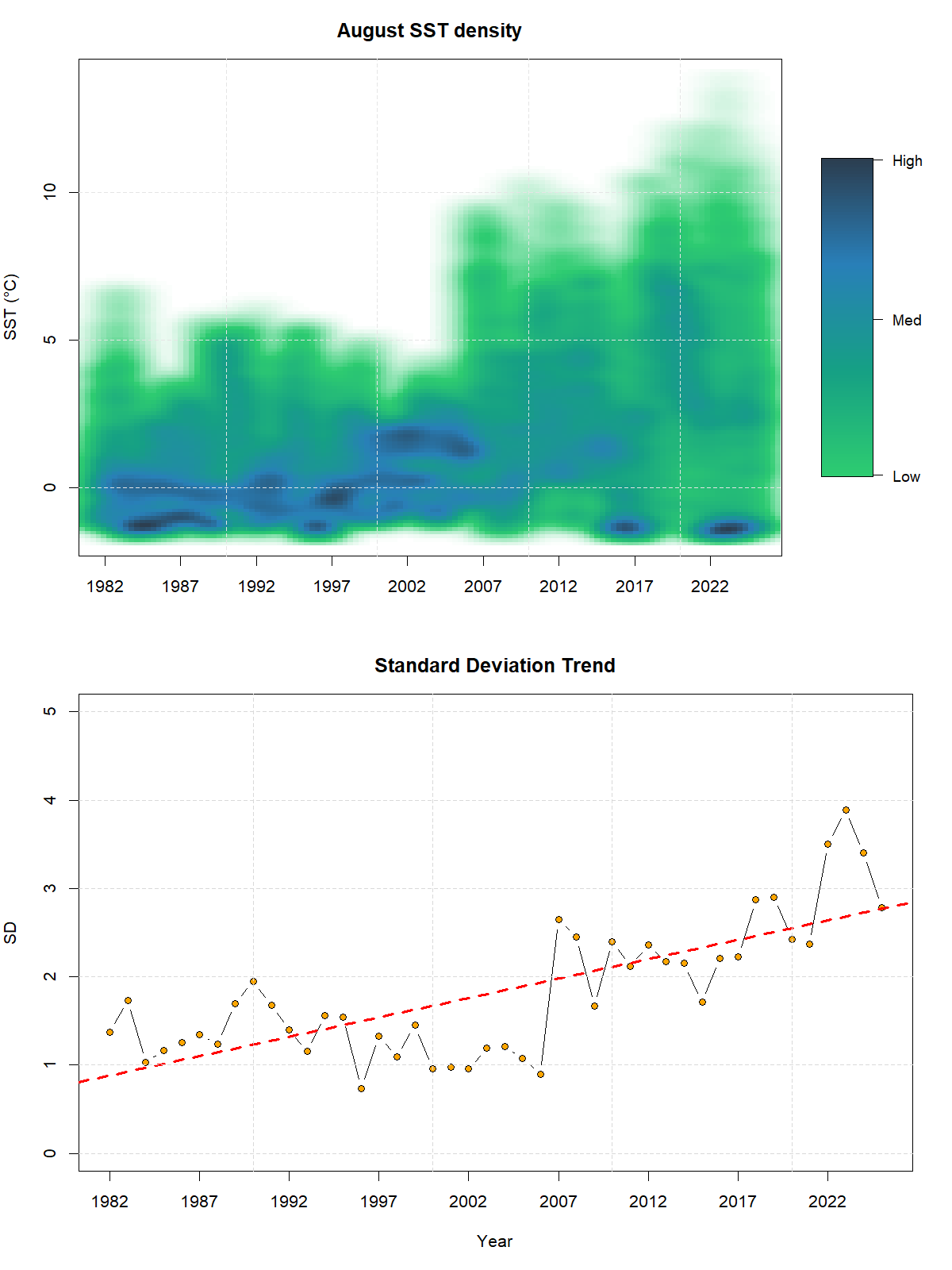}
    \caption{Laptev Sea (panel 4, Figure \ref{fig: area_arctic}). Top: density distribution of annual August mean SST values across all individual grid points of the study area. Bottom: interannual standard deviation (dashed line with markers), together with its growth trend (dashed red line).}
    \label{fig: var_spa_laptev}
\end{figure}

Regarding the spatial variability of temperature, as mentioned before, we changed dataset. For each grid point within the study area, we used the August mean SST value to create a density map, where the colour intensity represents the frequency of SST values over time. We then calculated the standard deviation for each year, thus capturing the area's spatial spread, the fluctuations around the average value, i.e. its heterogeneity.

Figures \ref{fig: var_spa_kara} and \ref{fig: var_spa_laptev} illustrate the spatial variability of the previously discussed seas, i.e. the Kara and Laptev Seas, respectively. The top panel shows the density map. In both cases, a clear widening of the spread is visible, indicating greater dispersion. The lower panel displays the interannual standard deviation, together with its growth trend (dashed red line), which is increasing.

\begin{figure}
    \centering
    \includegraphics[width=.9\linewidth]{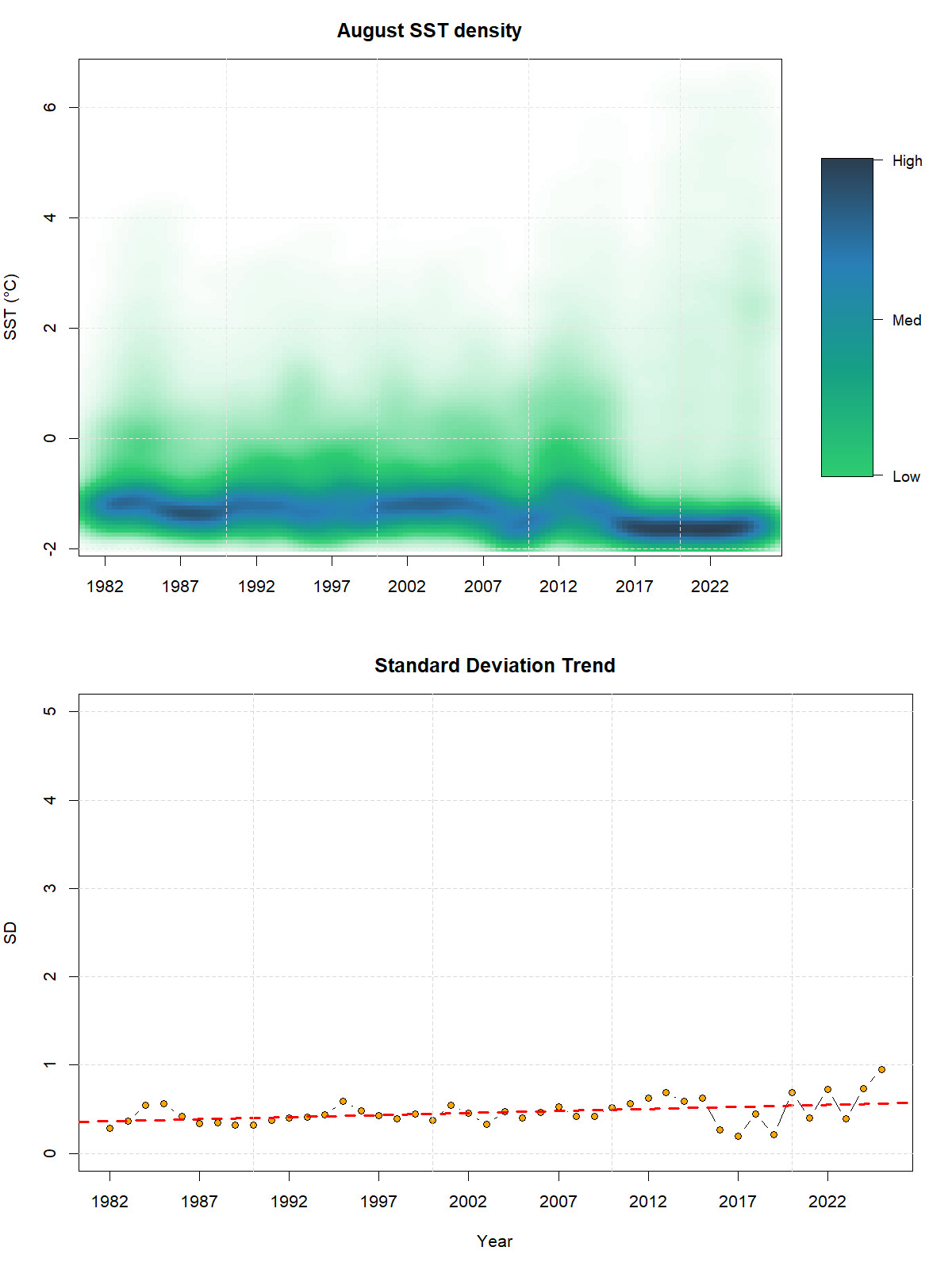}
    \caption{Arctic Ocean (panel 10, Figure \ref{fig: area_arctic}). Top: density distribution of annual August mean SST values across all individual grid points of the study area. Bottom: interannual standard deviation (dashed line with markers), together with its growth trend (dashed red line).}
    \label{fig: var_spa_arctic}
\end{figure}

In contrast to the aforementioned seas, the Arctic Ocean (panel 10, Figure \ref{fig: area_arctic}) exhibits a different behaviour. As shown in Figure \ref{fig: var_spa_arctic}, the standard deviation remains stationary. This could possibly be explained by the fact that the average temperatures of those seas are still too low; that is, we are referring to the left-side plateau of the co-albedo curve (equation \eqref{eq: piecewise co-albedo}), i.e. in \eqref{prop: variance derivative} $\mathrm{Var}_\infty'\equiv 0$. Indeed, persistent ice still characterises this area even during August.

\section*{Acknowledgments}
G.D.S. acknowledges the support from the DFG project FOR~5528. The research of F.F. is funded by the PRIN-PNRR project “Some mathematical
approaches to climate change and its impacts” n. P20225SP98 and the European Union (ERC, NoisyFluid, No. 101053472). Views and opinions expressed are however those of the authors only and do not necessarily reflect those of the European Union or the European Research Council. Neither the European Union nor the granting authority can be held responsible for them. M.L. is supported by the the Italian national inter-university PhD course in Sustainable Development and Climate change. M.L. produced this work while attending the PhD programme in
PhD in Sustainable Development And Climate Change at the University School for Advanced Studies IUSS
Pavia, Cycle XXXVIII, with the support of a scholarship financed by the Ministerial Decree no. 351 of 9th
April 2022, based on the NRRP - funded by the European Union - NextGenerationEU - Mission 4 "Education
and Research", Component 1 "Enhancement of the offer of educational services: from nurseries to
universities” - Investment 3.4 “Advanced teaching and university skills”. 


\end{document}